\documentclass{amsart}
\usepackage{amsfonts}

\newtheorem{theorem}{Theorem}
\theoremstyle{plain}

\newtheorem{definition}{Definition}
\newtheorem{example}{Example}

\newtheorem{lemma}{Lemma}

\newtheorem{proposition}{Proposition}
\newtheorem{remark}{Remark}

\numberwithin{equation}{section}
\input{tcilatex}

\begin{document}
\title[AG-groupoids characterized by $(\in ,\in \vee q_{k})$ fuzzy bi-ideals]%
{Abel-Grassmann's groupoids characterized by $(\in ,\in \vee q_{k})$ fuzzy
bi-ideals}
\author{}
\maketitle

\begin{center}
$^{1}$\textbf{Madad Khan, }$^{2}$\textbf{Kifayat Ullah}

\textbf{Department of Mathematics}

\textbf{COMSATS Institute of Information Technology,}

\textbf{Abbottabad}

\textrm{e-mails: }$^{1}$\textrm{madadmath@yahoo.com, }$^{2}$\textrm{%
kifi\_marwat949@yahoo.com}
\end{center}

\textbf{Abstract.} Using the idea of a quasi-coincedece of a fuzzy point
with a fuzzy set, the concept of an $(\alpha ,\beta )$-fuzzy bi-ibeals in
AG-groupoid is introduced in this paper, which is a generalization of the
concept of a fuzzy bi-ideal in AG-groupoid and some interesting
characterizations theorems are obtained.


\section{\protect \LARGE Introduction}

The fundamental concept of a fuzzy set was introduced by Zadeh in his paper 
\cite{L.A.Zadeh}, $1965$. For a given set $S$, a fuzzy subset $f$ of $S$ is
a mapping $f:S\longrightarrow \lbrack 0,1]$, where $[0,1]$ is the unit
interval. After taking the idea of fuzzification, a lot of research papers
has been published by different mathematicians in different fields of
mathematics, which shows the importance and application of fuzzy set to set
theory, group theory, semigroup theory, groupoids, real analysis, measure
theory and topology etc. Recently in \cite{Zadeh}, Zadeh give the
relationship of fuzzy set with that of probability.

Fuzzy set was applied to generalized some basic concepts of general topology
in \cite{cha}. The fuzzy group was firstly developed by Rosenfeld \cite{A.
Rosenfeld}. Using of fuzzy sets in semigroup was developed by Kuroki \cite%
{N. Kuroki 2}, and study fuzzy bi-ideals in semigroup \cite{N. Kuroki}.
Murali \cite{Murali} define the concept of belongingness of a fuzzy point to
a fuzzy subset \ under a natural equivalence on a fuzzy subset. In \cite{Pu}
the idea of quasi-coincidence of a fuzzy point with a fuzzy set is defined.
These ideas played important role in generating some different types of
fuzzy subgroups. Using these ideas Bhakat and Das in \cite{Bakht 1} and \cite%
{Bakht 2} gave the concept of $(\alpha ,\beta )$-fuzzy subgroups, where $%
\alpha ,\beta \in \{ \in ,q,\in \vee q,\in \wedge q\}$ and $\alpha \neq \in
\wedge q$. The concept of $(\in ,\in \vee q)$-fuzzy subgroups is viable
generalization of Rosenfeld's fuzzy subgroupoid. $(\in ,\in \vee q_{k})$%
-fuzzy ideals, $(\in ,\in \vee q_{k})$-fuzzy quasi-ideals and $(\in ,\in
\vee q_{k})$-fuzzy bi-ideals of a semigroup is defined in \cite{junn2}, also
characterization of different classes of semigroup by the properties of
these fuzzy ideals has been discussed in this paper.

In this paper we define $(\in ,\in \vee q)$ and $(\in ,\in \vee q_{k})$%
-fuzzy (ideals, bi-ideals, generalized bi-ideals, quasi ideals and interior
ideals) of an Abel-Grassmann's groupoid and characterize its different
classes by these fuzzy ideals. Abel-Grassmann's groupoid is a
non-associative algebraic structure midway between a groupoid and a
commutative semigroup.

The concept of AG-groupoid was first given by M. A. Kazim and M.\
Naseeruddin \cite{Kaz} in $1972$. An Abel-Grassmann's groupoid (AG-groupoid) 
\cite{protic} or A left almost semigroup (LA-semigroup) \cite{Kaz} is a
groupoid $S$ holding the following left invertive law 
\begin{equation}
(ab)c=(cb)a\text{, for\ all }a\text{, }b\text{, }c\in S\text{.}  \tag{$1$}
\end{equation}%
In an AG -groupoid, the medial law \cite{Kaz} holds, 
\begin{equation}
(ab)(cd)=(ac)(bd)\text{, for\ all }a\text{, }b\text{, }c\text{, }d\in S\text{%
.}  \tag{$2$}
\end{equation}%
The left identity in an AG -groupoid if exists is unique \cite{Mus3}.

In an AG -groupoid $S$ with left identity the paramedial law holds,%
\begin{equation}
(ab)(cd)=(db)(ca)\text{, for\ all }a\text{, }b\text{, }c\text{, }d\in S\text{%
.}  \tag{$3$}
\end{equation}

An AG -groupoid is a non-associative algebraic structure mid way between a
groupoid and a commutative semigroup with wide applications in the theory of
flocks. An AG -groupoid with right identity becomes a commutative semigroup
with identity \cite{Mus3}. If an AG-groupoid contains left identity, the
following law holds,

\begin{equation}
a(bc)=b(ac)\text{, for\ all }a\text{, }b\text{, }c\in S\text{.}  \tag{$4$}
\end{equation}

\section{\protect \LARGE Preliminaries }

Let $S$ be an AG-groupoid. By AG-subgroupiod of $S$ we means a non-empty
subset $A$ of $S$ such that $A^{2}\subseteq A$, and by a left (right) ideal
of $S$ we mean a non-empty subset $A$ of $S$ such that $SA\subseteq A$ ($%
AS\subseteq A$). By two-sided ideals or simply ideal, we mean a non-empty
subset of $S$ which is both a left and a right ideal of $S$. An
AG-subgroupiod $A$ of $S$ is called bi-ideal of $S$ if $(AS)A\subseteq A$. A
subset $A$ of $S$ is called generalized bi-ideal of $S$ if $(AS)A\subseteq A$%
. An AG-subgroupiod $A$ of $S$ is called interior ideal of $S$ if $%
(SA)S\subseteq A$.

A fuzzy subset $f$ of a given set $S$, as described above, is an arbitrary
function $f:S\longrightarrow \lbrack 0$,$1]$, where $[0$,$1]$ is the usual
closed interval of real numbers. For any two fuzzy subsets $f$ and $g$ of $S$%
, $f\subseteq g$ means that, $f(x)\leq g(x)$ for all $x$ in $S$. The symbols 
$f\cap g$ and $f\cup g$ will means the following fuzzy subsets of $S$%
\begin{eqnarray*}
(f\cap g)(x) &=&\min \{f(x),g(x)\}=f(x)\wedge g(x) \\
(f\cup g)(x) &=&\max \{f(x),g(x)\}=f(x)\vee g(x)
\end{eqnarray*}

for all $x$ in $S$.

Let $f$ and $g$ be any fuzzy subsets of an AG-groupoid $S$, then the product 
$f\circ g$ is defined by

\begin{equation*}
\left( f\circ g\right) (a)=\left \{ 
\begin{array}{c}
\dbigvee \limits_{a=bc}\left \{ f(b)\wedge g(c)\right \} \text{, if there
exist }b,\text{ }c\in S\text{, such that }a=bc \\ 
0,\text{ \  \  \  \  \  \  \  \  \  \  \  \  \  \  \  \  \  \  \  \  \  \  \  \  \  \  \  \  \  \  \  \  \  \
\  \  \  \  \  \  \  \  \  \  \  \  \  \ otherwise.}%
\end{array}%
\right.
\end{equation*}

Let $A$ be a subset of an AG-groupoid $S$, then the characteristic function
of $A$, that is $C_{A}$ is defined by

\begin{center}
$C_{A}(x)=\left \{ 
\begin{array}{c}
1\text{, if }x\in A, \\ 
0\text{, if }x\notin A.%
\end{array}%
\right. $
\end{center}

\begin{definition}
\label{d 1}$(i)$A fuzzy subset $f$ of an AG-groupoid $S$ is called a fuzzy
AG-subgroupoid of $S$ if $f(xy)\geq f(x)\wedge f(y)$ for all $x$, $y\in S.$

$(ii)$A fuzzy AG-subgroupoid $f$ of an AG-groupoid $S$ is called fuzzy
bi-ideal of $S$ if $f((xy)z)\geq f(x)\wedge f(z)$, for all $x$, $y$ and $%
z\in S$.

$(iii)$A fuzzy subset $f$ of an AG-groupoid $S$ is called fuzzy generalized
bi-ideal of $S$ if $f((xy)z)\geq f(x)\wedge f(z)$, for all $x$, $y$ and $%
z\in S$.
\end{definition}

\begin{definition}
Let $f$ be any fuzzy subsets of an AG-groupoid $S$, then for all $t\in (0,1]$%
, the set $f_{t}=\{x\in S\mid f(x)\geq t\}$ is called a level subset of $S$.

\begin{theorem}
Let $S$ be an AG-groupoid, and $f$ be a fuzzy subset of $S$. Then $f$ is a
fuzzy bi-ideal of $S$ if and only if the level subset $f_{t}(\neq \{ \})$ is
a bi-ideal of $S$ for all $t\in (0,1]$.
\end{theorem}
\end{definition}

Let $F(S)$ denotes the set of all fuzzy subsets of $S$, then $(F(S),\circ )$
is an AG-groupiod \cite{Madad Khan}.

The proofs of the following lemmas are available in \cite{Nouman}.

\begin{lemma}
A non-empty subset $B$ of an AG-groupoid $S$ is bi-ideal of $S$ if and only
if the $C_{B}$ is a fuzzy bi-ideal of $S$.
\end{lemma}

\begin{lemma}
Let $A$ and $B$ be any non-empty subsets of an AG-groupoid $S$, then the
following properties hold.

$(i)$ $C_{A}\cap C_{B}=C_{A\cap B}$,

$(ii)$ $C_{A}\circ C_{B}=C_{AB}$.
\end{lemma}

\section{$(\protect \alpha ,\protect \beta )$-fuzzy bi-ibeals}

From this section, for a fuzzy subset $f$ of an AG-groupiod $S$ and $t\in
(0,1],$ we have the following new concepts,

\begin{enumerate}
\item $x_{t}\in f$ means $f(x)\geq t$,

\item $x_{t}qf$ means $f(x)+t>1$,

\item $x_{t}\alpha \vee \beta f$ means $x_{t}\alpha f$ or $x_{t}\beta f$,

\item $x_{t}\alpha \wedge \beta f$ means $x_{t}\alpha f$ and $x_{t}\beta f$,

\item $x_{t}\overline{\alpha }f$ means $x_{t}\alpha f$ does not holds.
\end{enumerate}

In what follows let $S$ denotes an AG-groupoid and $\alpha ,\beta $ denotes
any one of $\in ,$ $q,$ $\in \vee q,$ $\in \wedge q$ unless otherwise
specified.

\begin{definition}
\label{def 1}Let $S$ be an AG-groupoid, and $f$ be a fuzzy subset of $S$.
Then $f$ is an $(\alpha ,\beta )$-fuzzy subgroupoid of $S$, if for all $x$, $%
y\in S$ and $t$, $r\in (0,1]$, we have $x_{t}\alpha f$, and $y_{r}\alpha
f\Longrightarrow (xy)_{t\wedge r}\beta f$.
\end{definition}

\begin{definition}
\label{def 2}Let $S$ be an AG-groupoid, and $f$ be a fuzzy subset of $S$.
Then $f$ is an $(\alpha ,\beta )$-fuzzy generalized bi-ibeal of $S$, if for
all $x$, $y$, $z\in S$ and $t$, $r\in (0,1]$, we have $x_{t}\alpha f$ and $%
z_{r}\alpha f\Longrightarrow ((xy)z)_{t\wedge r}\beta f.$
\end{definition}

\begin{definition}
Let $S$ be an AG-groupoid, and $f$ be a fuzzy subset of $S$. Then $f$ is an $%
(\alpha ,\beta )$-fuzzy bi-ibeal of $S$ if definition \ref{def 1} and \ref%
{def 2} holds.
\end{definition}

Every fuzzy bi-ideal of $S$ is an $(\in ,\in )$-fuzzy bi-ibeal of $S$ as
shown in the following theorem.

\begin{theorem}
\label{T 3.1}For a fuzzy subset $f$ of an AG-groupiod $S$, the following are
equivalent,

$(i)$ $f$ is a fuzzy bi-ideal of $S$

$(ii)$ $f$ is an $(\in ,\in )$-fuzzy bi-ibeals.
\end{theorem}

\begin{proof}
$(i)\Longrightarrow (ii)$

Let $x,y\in S$ and $t,r\in (0,1]$ be such that $x_{t}$, $y_{r}\in f.$ Then $%
f(x)\geq t$ and $f(y)\geq r.$ Now by definition \ref{d 1}$(i)$%
\begin{equation*}
f(xy)\geq f(x)\wedge f(y)\geq t\wedge r\text{,}
\end{equation*}

implies that $(xy)_{t\wedge r}\in f,$ and let $x$, $y$, $z\in S$ and $t,r\in
(0,1]$ be such that $x_{t}$, $z_{r}\in f$. Then$.f(x)\geq t$ and $f(z)\geq
r. $ Now by definition \ref{d 1}$(ii)$%
\begin{equation*}
f((xy)z)\geq f(x)\wedge f(z)\geq t\wedge r\text{,}
\end{equation*}

which implies that $((xy)z)_{t\wedge r}\in f$. Therefore $f$ is an $(\in
,\in )$-fuzzy bi-ibeals.

$(ii)\Longrightarrow (i)$

Let $x,y\in S.$ Since $x_{f(x)}\in f$ and $y_{f(y)}\in f,$ since $f$ is an $%
(\in ,\in )$-fuzzy bi-ibeals, so $(xy)_{f(x)\wedge f(y)}\in f$, it follows
that $f(xy)\geq f(x)\wedge f(y)$, and let $x$, $y$, $z\in S$. Since $%
x_{f(x)}\in f$ and $z_{f(z)}\in f$ and $f$ is an $(\in ,\in )$-fuzzy
bi-ibeals so $((xy)z)_{f(x)\wedge f(z)}\in f$, it follows that $f((xy)z)\geq
f(x)\wedge f(z)$, so $f$ is a fuzzy bi-ideal of $S$.
\end{proof}

\begin{theorem}
Let $f$ be a non-zero $(\alpha ,\beta )$-fuzzy bi-ideal of $S$. Then the set 
$f_{0}=\{x\in S\mid f(x)>0\}$ is a bi-ibeal of $S$.
\end{theorem}

\begin{proof}
Let $x$, $y\in f_{0}\subseteq S,$ then $f(x)>0$ and $f(y)>0$. Assume that $%
f(xy)=0$. If $\alpha \in \{ \in ,\in \vee q\}$ then $x_{f(x)}\alpha f$ and $%
y_{f(y)}\alpha f$ but $(xy)_{f(x)\wedge f(y)}\overline{\beta }f$ for every $%
\beta \in \{ \in ,q,\in \vee q,\in \wedge q\}$, a contradiction. Note that $%
x_{1}qf$ and $y_{1}qf$ but $(xy)_{1\wedge 1}=(xy)_{1}\overline{\beta }f$ for
every $\beta \in \{ \in ,q,\in \vee q,\in \wedge q\}$, a contradiction.
Hence $f(xy)>0,$ that is $xy\in f_{0}$. Now let $x$, $z\in f_{0}$ and $y\in
S,$ $f(x)>0$ and $f(z)>0$. Assume that $f((xy)z)=0$. If $\alpha \in \{ \in
,\in \vee q\}$ then $x_{f(x)}\alpha f$ and $z_{f(z)}\alpha f$ but $%
((xy)z)_{f(x)\wedge f(z)}\overline{\beta }f$ for every $\beta \in \{ \in
,q,\in \vee q,\in \wedge q\}$, a contradiction. Note that $x_{1}qf$ and $%
z_{1}qf$ but $((xy)z)_{1\wedge 1}=((xy)z)_{1}\overline{\beta }f$ for every$%
\beta \in \{ \in ,q,\in \vee q,\in \wedge q\}$, a contradiction. Hence $%
f((xy)z)>0$, that is $(xy)z\in f_{0}$. Consequently, $f_{0}$\ is a bi-ideal
of $S$.
\end{proof}

\begin{theorem}
Let $f$ be a fuzzy subset and $B$ is a bi ideal of $S$ such that $f(x)=0$
for all $x\in S\backslash B$ and $f(x)\geq 0.5$ for all $x\in B$. Then

$(i)$ $f$ is a $(q,\in \vee q)$-fuzzy bi-ideal of $S,$

$(ii)$ $f$ is a $(\in ,\in \vee q)$-fuzzy bi-ideal of $S$.
\end{theorem}

\begin{proof}
$(i)$ Let $x,y\in S$ and $r,t\in (0,1]$ be such that $x_{r}qf$ and $y_{t}qf.$
Then $x,y\in B$ and we have $xy\in B.$

If $r\wedge t\leq 0.5,$ then $f(xy)\geq 0.5\geq r\wedge t$ and hence $%
(xy)_{r\wedge t}\in f.$ If $r\wedge t>0.5,$ then $f(xy)+r\wedge t>0.5+0.5=1$
and so $(xy)_{r\wedge t}qf$. Therefore $(xy)_{r\wedge t}\in \vee qf$. Now
let $x,$ $y,$ $z\in S$ and $r,t\in (0,1]$ be such that $x_{r}qf$ and $%
z_{t}qf.$ Then $x,z\in B$ and we have $((xy)z)\in B$. If $r\wedge t\leq 0.5,$
then $f((xy)z)\geq 0.5\geq r\wedge t$ and hence $((xy)z)_{r\wedge t}\in f.$
If $r\wedge t>0.5,$ then $f((xy)z)+r\wedge t>0.5+0.5=1$ and so $%
((xy)z)_{r\wedge t}qf$. Therefore $((xy)z)_{r\wedge t}\in \vee qf$. There
fore $f$ is a $(q,\in \vee q)-$fuzzy bi-ideal of $S$.

$(ii)$ Let $x,y\in S$ and $r,t\in (0,1]$ be such that $x_{r}\in f$ and $%
y_{t}\in f.$ Then $x,y\in B$ and we have $xy\in B$. If $r\wedge t\leq 0.5,$
then $f(xy)\geq 0.5\geq r\wedge t$ and hence $(xy)_{r\wedge t}\in f.$ If $%
r\wedge t>0.5,$ then 
\begin{equation*}
f(xy)+r\wedge t>0.5+0.5=1\text{.}
\end{equation*}
Therefore $(xy)_{r\wedge t}qf$. Thus $(xy)_{r\wedge t}\in \vee qf$. Now let $%
x,$ $y,$ $z\in S$ and $r,t\in (0,1]$ be such that $x_{r}\in f$ and $z_{t}\in
f.$ Then $x,z\in B$ and we have $((xy)z)\in B$. If $r\wedge t\leq 0.5,$ then%
\begin{equation*}
f((xy)z)\geq 0.5\geq r\wedge t
\end{equation*}

and hence $((xy)z)_{r\wedge t}\in f.$ If $r\wedge t>0.5,$ then%
\begin{equation*}
f((xy)z)+r\wedge t>0.5+0.5=1
\end{equation*}

and so $((xy)z)_{r\wedge t}qf$. Therefore $((xy)z)_{r\wedge t}\in \vee qf$.
Thus $f$ is a $(\in ,\in \vee q)$-fuzzy bi-ideal of $S$.
\end{proof}

\begin{theorem}
\label{T 2.1}Let $S$ is an AG-groupoid and $f$ be a fuzzy subset of $S$,
then $f$ is fuzzy bi-ideal of $S$ if and only if the non-empty level subset $%
f_{t}=\{x\in S\mid f(x)\geq t\}$ of $S$ is a bi-ideal of $S$ for all $t\in
(0,1]$.
\end{theorem}

\begin{proof}
It is straight forward.
\end{proof}

\section{$(\in ,\in \vee q)$-fuzzy bi-ideals}

\begin{definition}
If we put $\alpha =\in $ and $\beta =\in \vee q$ in definitions \ref{def 1}
and \ref{def 2}, then we get the $(\in ,\in \vee q)$-fuzzy bi-ideals. The
equivalent definition are given in the following lemmas, whose proofs is
same as given in \cite{junn}.
\end{definition}

\begin{lemma}
Let $f$ be a fuzzy subset of AG-groupoid $S$, then $f$ is $(\in ,\in \vee q)$%
-fuzzy subgroupoid of $S$ if%
\begin{equation*}
f(xy)\geq \min \{f(x),f(y),0.5\} \text{ for all }x,y\in S.
\end{equation*}
\end{lemma}

\begin{lemma}
Let $f$ be a fuzzy subset of AG-groupoid $S$, then $f$ is $(\in ,\in \vee q)$%
-fuzzy left(respectively right) ideal of $S$ if%
\begin{equation*}
f(xy)\geq \min \{f(y),0.5\}(\text{respectively }f(xy)\geq \min \{f(x),0.5\})%
\text{ for all }x,y\in S).
\end{equation*}
\end{lemma}

\begin{lemma}
Let $f$ be a fuzzy subset of AG-groupoid $S$, then $f$ is $(\in ,\in \vee q)$%
-fuzzy bi-ideals of $S$ if and only if

$(i)$ $f(xy)\geq \min \{f(x),$ $f(y),0.5\},$ for all $x,y\in S,$

$(ii)$ $f((xy)z)\geq \min \{f(x),f(z),0.5\},$ for all $x,y,z\in S$.
\end{lemma}

\begin{remark}
Every fuzzy bi-ideal if an AG-groupoid $S$ is $(\in ,\in \vee q)$-fuzzy
bi-ideals of $S$ but the converse is not true, in general.
\end{remark}

\begin{example}
Let $S=\left \{ a,b,c,d,e\right \} $ be an AG-groupoid with left identity $d$
with the following multiplicative table,

\begin{equation*}
\begin{tabular}{l|lllll}
$.$ & $a$ & $b$ & $c$ & $d$ & $e$ \\ \hline
$a$ & $a$ & $a$ & $a$ & $a$ & $a$ \\ 
$b$ & $a$ & $b$ & $b$ & $b$ & $b$ \\ 
$c$ & $a$ & $b$ & $d$ & $e$ & $c$ \\ 
$d$ & $a$ & $b$ & $c$ & $d$ & $e$ \\ 
$e$ & $a$ & $b$ & $e$ & $c$ & $d$%
\end{tabular}%
\end{equation*}

Clearly $\{a\}$ and $\{a,b\}$ are bi-ideals of $S$. Let a fuzzy subset $f$
be define as 
\begin{equation*}
f(a)=0.8,\text{ }f(b)=0.7,\text{ }f(c)=f(d)=f(e)=0.3\text{.}
\end{equation*}

Then 
\begin{equation*}
f_{t}=\left \{ 
\begin{array}{c}
S\text{ \  \  \  \  \  \  \  \  \ if }t\in (0,0.3] \\ 
\{a,b\} \text{ \  \  \  \  \ if }t\in (0.3,0.7] \\ 
\{a\} \text{ \  \  \  \  \  \  \  \ if }t\in (0.7,0.8] \\ 
\{ \} \text{ \  \  \  \  \  \  \  \ if }t\in (0.8,1]\text{.}%
\end{array}%
\right.
\end{equation*}

Clearly $f_{t}$ is bi-ideal for all $t\in (0,1]$. So by using theorem \ref{T
2.1}, $f$ is fuzzy bi-ideal of $S$ and hence $(\in ,\in \vee q)$-fuzzy
bi-ideals of $S$.
\end{example}

\begin{definition}
Let $f$ and $g$ be a fuzzy subsets of AG-groupoid $S$, then the $0.5$%
-product of $f$ and $g$ is defined by%
\begin{equation*}
(f\circ _{0.5}g)(a)=\left \{ 
\begin{array}{c}
\dbigvee \limits_{a=bc}\min \{f(a),\text{ }f(b),\text{ }0.5\} \  \text{if
there exists }b,c\in S\text{ such that }a=bc, \\ 
0\text{ \  \  \  \  \  \  \  \  \  \  \  \  \  \  \  \  \  \  \  \  \  \  \  \  \  \  \  \  \  \  \  \  \  \
\  \  \  \  \  \  \  \  \  \  \  \  \  \  \  \  \  \  \  \  \  \ otherwise.}%
\end{array}%
\right.
\end{equation*}

The $0.5$ intersection of $f$ and $g$ is defined by 
\begin{equation*}
(f\cap _{0.5}g)(a)=\{f(a)\wedge f(b)\wedge 0.5\} \text{ for all }a\in S\text{%
.}
\end{equation*}
\end{definition}

\begin{proposition}
\label{P 4.2}Let $S$ be an AG-groupoid and $f,g,h,k$ are fuzzy subsets of $S$
such that $f\subseteq h$ and $g\subseteq k$, then $f\circ _{0.5}g\subseteq
h\circ _{0.5}k$.
\end{proposition}

\begin{proposition}
\label{P 4.3}Let $f$ and $g$ are $(\in ,\in \vee q)$-fuzzy bi-ideals of
AG-groupoid $S$, then $f\cap _{0.5}g$ is also $(\in ,\in \vee q)$-fuzzy
bi-ideal of $S.$
\end{proposition}

\begin{lemma}
\label{L 4.2}Let $S$ be an AG-groupoid. Then every one-sided $(\in ,\in \vee
q)$-fuzzy ideal is an $(\in ,\in \vee q)$-fuzzy bi-ideal of $S.$
\end{lemma}

\begin{proof}
Let $f$ be an $(\in ,\in \vee q)$-fuzzy left ideal of $S$ and $x,y\in S.$
Then $f(xy)\geq \min \{f(y),$ $0.5\} \geq \min \{f(x),$ $f(y),$ $0.5\}.$
Therefore $f$ is an $(\in ,\in \vee q)$-fuzzy subgroupoid of $S.$ Let $%
x,y,z\in S.$ Then $f((xy)z)\geq \min \{f(z),$ $0.5\} \geq \min \{f(x),$ $%
f(z), $ $0.5\}.$ Thus $f$ is an $(\in ,\in \vee q)$-fuzzy bi-ideal of $S.$
Now let $f$ be an $(\in ,\in \vee q)$-fuzzy right ideal of $S$ and $x,y\in
S. $ Then $f(xy)\geq \min \{f(x),$ $0.5\} \geq \min \{f(x),$ $f(y),$ $0.5\}.$
Hence $f$ is an $(\in ,\in \vee q)$-fuzzy subgroupoid of $S.$ Let $x,y,z\in
S.$ Then $f((xy)z)\geq \min \{f(xy),$ $0.5\} \geq \{f(x),$ $0.5\} \geq \min
\{f(x),$ $f(z),$ $0.5\}.$

Hence $f$ is an $(\in ,\in \vee q)$-fuzzy bi-ideal of $S.$
\end{proof}

\begin{definition}
An $(\in ,\in \vee q)$-fuzzy bi-ideal $f$ of $S$ is called idempotent if $%
f\circ _{0.5}f=f.$
\end{definition}

\begin{proposition}
\label{P 4.4}Let $S$ be an AG-groupoid and $f$ is an $(\in ,\in \vee q)$-
fuzzy bi-ideal of $S$, then $f\circ _{0.5}f\subseteq f.$
\end{proposition}

\begin{proof}
Let $f$ be an $(\in ,\in \vee q)$-fuzzy bi-ideals of $S$ and let $a\in S$.
If $a\neq bc$ for some $b,c\in S$, then $(f\circ _{0.5}f)(a)=0\leq f(a)\,$,
and if $a=bc$ for some $b,c\in S$, then 
\begin{equation*}
(f\circ _{0.5}f)(a)=\dbigvee \limits_{a=bc}\min \{f(b),f(c),0.5\} \leq
\dbigvee \limits_{a=bc}f(bc)\leq \dbigvee \limits_{a=bc}f(a)=f(a).
\end{equation*}

Hence $f\circ _{0.5}f\subseteq f.$
\end{proof}

\begin{definition}
For an AG-groupoid the fuzzy subset $0$(respectively $1$) is defined as
follows: $0(x)=0$($1(x)=1$) for all $x\in S.$
\end{definition}

\begin{proposition}
For any fuzzy subset $f$ of an AG-groupoid we have $f\subseteq 1.$
\end{proposition}

\begin{proof}
It is straight forward.
\end{proof}

\begin{lemma}
Let $S$ be an AG-groupoid and $f,g$ are fuzzy subsets of $S$. Then 
\begin{equation*}
f\circ _{0.5}g\subseteq 1\circ _{0.5}g(\text{respectively }f\circ
_{0.5}g\subseteq f\circ _{0.5}1).
\end{equation*}
\end{lemma}

\begin{proof}
It follows from proposition.\ref{P 4.2}.
\end{proof}

\begin{proposition}
\label{P 4.5}Let $S$ be an AG-groupoid and $f$ is an $(\in ,\in \vee q)$-
fuzzy bi-ideal of $S$, then $(f\circ _{0.5}1)\circ _{0.5}f\subseteq f.$
\end{proposition}

\begin{proof}
Let $f$ be an $(\in ,\in \vee q)$-fuzzy bi-ideal of $S$ and let $a\in S$. If 
$a\neq bc$ for some $b,c\in S$, then $((f\circ _{0.5}1)\circ
_{0.5}f)(a)=0\leq f(a)$, and if $a=bc$ for some $b,c\in S$, then 
\begin{eqnarray*}
((f\circ _{0.5}1)\circ _{0.5}f)(a) &=&\dbigvee \limits_{a=bc}\min \{(f\circ
_{0.5}1)(b),f(c),0.5\} \\
&=&\dbigvee \limits_{a=bc}\min \{ \dbigvee \limits_{b=tr}\min
\{f(t),1(r),0.5\},~f(c),0.5\} \\
&=&\dbigvee \limits_{a=bc}\dbigvee \limits_{b=tr}\min \{f(t),f(c),1,0.5\} \\
&=&\dbigvee \limits_{a=bc}\dbigvee \limits_{b=tr}\min \{f(t),f(c),0.5\} 
\text{.}
\end{eqnarray*}

Since $a=bc=(tr)c$, and $f$ is an $(\in ,\in \vee q)$- fuzzy bi-ideal of $S$%
, so we have

$f((tr)c)\geq \min \{f(t),f(c),0.5\}$. Thus 
\begin{eqnarray*}
((f\circ _{0.5}1)\circ _{0.5}f)(a) &=&\dbigvee \limits_{a=bc}\dbigvee
\limits_{b=tr}\min \{f(t),f(c),0.5\} \\
&\leq &\dbigvee \limits_{a=(tr)c}f((tr)c)=f(a)\text{,}
\end{eqnarray*}

which implies that $(f\circ _{0.5}1)\circ _{0.5}f\subseteq f.$
\end{proof}

\begin{definition}
$(i)$ An AG-groupoid $S$ is called regular, if for each $a\in S$ there
exists $x\in S,~$such that $a=(ax)a$.

$(ii)$ An AG-groupoid $S$ is called intra-regular, if for each $a\in S$
there exists $x,~y\in S,$ such that $a=(xa^{2})y$.
\end{definition}

\begin{theorem}
\label{T 4.3}For a regular AG-groupoid $S$ with left identity $e$, we have $%
(f\circ _{0.5}1)\circ _{0.5}f=f$ for every $(\in ,\in \vee q)$-fuzzy
bi-ideal $f$ of $S$.
\end{theorem}

\begin{proof}
Let $S$ is regular AG-groupoid and let $a\in S$, then there exists $x\in S,$
such that $a=(ax)a$. Now by using $(1)$ and $(4)$, we have%
\begin{equation*}
a=(ax)a=(((ax)a)a)a=((aa)(ax))a=(a((aa)x))a,
\end{equation*}

therefore

\begin{eqnarray*}
((f\circ _{0.5}1)\circ _{0.5}f)(a) &=&\dbigvee \limits_{a=bc}\min \{(f\circ
_{0.5}1)(b),f(c),0.5\} \\
&\geq &\min \{(f\circ _{0.5}1)(a((aa)x)),f(a),0.5\} \\
&=&\min \{ \dbigvee \limits_{a((aa)x)=tr}\min \{f(t),1(r),0.5\},~f(a),0.5\}
\\
&\geq &\min \{ \min \{f(a),1((aa)x),0.5\},f(a),0.5\} \\
&=&\min \{f(a),1,0.5\}=\min \{f(a),0.5\}=f(a).
\end{eqnarray*}

Hence $f\subseteq (f\circ _{0.5}1)\circ _{0.5}f$.\ On the other hand we know
by preposition \ref{P 4.5} that $(f\circ _{0.5}1)\circ _{0.5}f\subseteq f$.
This implies that $(f\circ _{0.5}1)\circ _{0.5}f=f$.
\end{proof}

\begin{lemma}
\label{L 4.4}Let $S$ be a regular AG-groupoid with left identity $e$ and $%
f,g $ be $(\in ,\in \vee q)$- fuzzy bi-ideals of $S$, then $f\circ _{0.5}g$
is also an $(\in ,\in \vee q)$- fuzzy bi-ideal of $S$.
\end{lemma}

\begin{proof}
Let $f$ and $g$ be an $(\in ,\in \vee q)$-fuzzy bi-ideals of $S$ and let $%
a\in S$.

If $a\neq bc$ for some $b,c\in S$, then $((f\circ _{0.5}g)\circ
_{0.5}(f\circ _{0.5}g))(a)=0\leq (f\circ _{0.5}g)(a)\,$,

and if $a=bc$ for some $b,c\in S$, then%
\begin{eqnarray*}
((f\circ _{0.5}g)\circ _{0.5}(f\circ _{0.5}g))(a) &=&\dbigvee
\limits_{a=bc}\{(f\circ _{0.5}g)(b)\wedge (f\circ _{0.5}g)(c)\wedge 0.5\} \\
&=&\dbigvee \limits_{a=bc}\left \{ 
\begin{array}{c}
\dbigvee \limits_{b=kl}\{f(k)\wedge g(l)\wedge 0.5\} \\ 
\wedge \dbigvee \limits_{c=mn}\{f(m)\wedge g(n)\wedge 0.5\} \wedge 0.5%
\end{array}%
\right \} \\
&=&\dbigvee \limits_{a=bc}\dbigvee \limits_{b=kl}\dbigvee
\limits_{c=mn}\{f(k)\wedge g(l)\wedge f(m)\wedge g(n)\wedge 0.5\} \\
&\leq &\dbigvee \limits_{a=bc}\dbigvee \limits_{b=kl}\dbigvee
\limits_{c=mn}\{f(m)\wedge g(n)\wedge 0.5\} \text{.}
\end{eqnarray*}

Since $a=bc$, $b=kl$ and $c=mn$. So $a=(kl)(mn)$. Also since $S$ is regular
so there exists $x\in S$ such that $a=(ax)a$. Now by using $(4),(2),(3)$ and 
$(1)$ we have%
\begin{eqnarray*}
a &=&(ax)a=(((kl)(mn))x)((kl)(mn)) \\
&=&(((kl)(mn))x)(m((kl)n))=m((((kl)(mn))x)((kl)n)) \\
&=&m((((kl)(mn))(ex))((kl)n))=m(((xe)((mn)(kl)))((kl)n)) \\
&=&m(((xe)((lk)(nm)))((kl)n))=m(((xe)(n((lk)m)))((kl)n)) \\
&=&m((n((xe)((lk)m)))((kl)n))=m((((kl)n)((xe)((lk)m)))n) \\
&=&m(((((lk)m)(xe))(n(kl)))n)=m((n((((lk)m)(xe))(kl)))n)\text{.}
\end{eqnarray*}

Then 
\begin{eqnarray*}
&&((f\circ _{0.5}g)\circ _{0.5}(f\circ _{0.5}g))(a) \\
&\leq &\dbigvee \limits_{a=bc}\dbigvee \limits_{b=kl}\dbigvee
\limits_{c=mn}\{f(m)\wedge g(n)\wedge 0.5\} \\
&\leq &\dbigvee \limits_{a=m((n((((lk)m)(xe))(kl)))n)}\{f(m)\wedge
g(n)\wedge 0.5\} \text{.}
\end{eqnarray*}

Since $g$ is an $(\in ,\in \vee q)$- fuzzy bi-ideal of $S$ so we have 
\begin{eqnarray*}
g((n((((lk)m)(xe))(kl)))n) &\geq &\{g(n)\wedge g(n)\wedge 0.5\} \\
&\geq &\{g(n)\wedge 0.5\} \text{.}
\end{eqnarray*}

Then 
\begin{eqnarray*}
&&((f\circ _{0.5}g)\circ _{0.5}(f\circ _{0.5}g))(a) \\
&\leq &\dbigvee \limits_{a=m((n((((lk)m)(xe))(kl)))n)}\{f(m)\wedge
g(n)\wedge 0.5\} \\
&=&\dbigvee \limits_{a=m((n((((lk)m)(xe))(kl)))n)}\{f(m)\wedge g(n)\wedge
0.5\wedge 0.5\} \\
&\leq &\dbigvee \limits_{a=m((n((((lk)m)(xe))(kl)))n)}\{f(m)\wedge
g((n((((lk)m)(xe))(kl)))n)\wedge 0.5\} \\
&\leq &\dbigvee \limits_{a=uv}\{f(u)\wedge g(v)\wedge 0.5\}=(f\circ
_{0.5}g)(a)\text{.}
\end{eqnarray*}

Therefore $((f\circ _{0.5}g)\circ _{0.5}(f\circ _{0.5}g))(a)\leq (f\circ
_{0.5}g)(a),$ and $f\circ _{0.5}g$ is an $(\in ,\in \vee q)$- fuzzy
subgroupoid of $S$.

Let $a,b,c\in S$. Let $a=pq$ and $c=rs$. Then by using $(2),(3),(1)$ and $%
(4) $ we have%
\begin{equation*}
(ab)c=((pq)b)(rs)=(sr)(b(pq))=((b(pq))r)s=((p(bq))r)s\text{,}
\end{equation*}

and since $f$ is an $(\in ,\in \vee q)$- fuzzy bi-ideal of $S$, so, $%
f((p(bq))r)\geq \{f(p)\wedge f(r)\wedge 0.5\}$. Now%
\begin{eqnarray*}
(f\circ _{0.5}g)(a)\wedge (f\circ _{0.5}g)(c) &=&\left[ \dbigvee%
\limits_{a=pq}\{f(p)\wedge g(q)\wedge 0.5\} \right] \wedge \left[ \dbigvee
\limits_{c=rs}\{f(r)\wedge g(s)\wedge 0.5\} \right] \\
&=&\dbigvee \limits_{a=pq}\dbigvee \limits_{c=rs}[\{f(p)\wedge g(q)\wedge
0.5\} \wedge \{f(r)\wedge g(s)\wedge 0.5\}] \\
&=&\dbigvee \limits_{a=pq}\dbigvee \limits_{c=rs}\{f(p)\wedge f(r)\wedge
g(q)\wedge g(s)\wedge 0.5\} \\
&\leq &\dbigvee \limits_{a=pq}\dbigvee \limits_{c=rs}\{f(p)\wedge f(r)\wedge
g(s)\wedge 0.5\} \\
&\leq &\dbigvee \limits_{a=((p(bq))r)s}\{f(p)\wedge f(r)\wedge g(s)\wedge
0.5\} \\
&\leq &\dbigvee \limits_{a=((p(bq))r)s}\{f((p(bq))r)\wedge g(s)\wedge
0.5\}=(f\circ _{0.5}g)(a)\text{.}
\end{eqnarray*}%
Thus $(f\circ _{0.5}g)((ab)c)\geq (f\circ _{0.5}g)(a)\wedge (f\circ
_{0.5}g)(c)$. Consequently $f\circ _{0.5}g$ is an $(\in ,\in \vee q)$- fuzzy
bi-ideal of $S$.
\end{proof}

\begin{theorem}
\label{T 4.4}Let $S$ be a regular and intra-regular AG-groupoid with left
identity $e$, then the following holds:

$(i)$ $f\circ _{0.5}f=f$ for every $(\in ,\in \vee q)$-fuzzy bi-ideal $f$ of 
$S$,

$(ii)$ $f$ $\cap _{0.5}g=$ $f\circ _{0.5}g\cap _{0.5}g\circ _{0.5}f,$ for
all $(\in ,\in \vee q)$-fuzzy bi-ideals $f$ and $g$ of S.
\end{theorem}

\begin{proof}
$(i)$

Let $f$ be an $(\in ,\in \vee q)$-fuzzy bi-ideal of $S$ and $a\in S$. Since $%
S$ is regular and intra-regular, so there exists $x,~y,~z\in S,$ such that $%
a=(ax)a$ and $a=(ya^{2})z,$now by using $(2),$ $(3)$, $(4)$and $(1)$ we have%
\begin{eqnarray*}
a &=&(ax)a=(((ax)a)x)((ax)a)=(((ax)((ya^{2})z))x)((ax)a) \\
&=&(((ax)((ya^{2})(ez)))x)((ax)a)=(((ax)((ze)(a^{2}y)))x)((ax)a) \\
&=&(((ax)(a^{2}((ze)y)))x)((ax)a)=((a^{2}((ax)((ze)y)))x)((ax)a) \\
&=&((x((ax)((ze)y)))a^{2})((ax)a)=((x((ze)((ax)y)))a^{2})((ax)a) \\
&=&((x((z(ax))y))a^{2})((ax)a)=(((z(ax))(xy))a^{2})((ax)a) \\
&=&(((zx)((ax)y))a^{2})((ax)a)=(((zx)((((ya^{2})z)x)y))a^{2})((ax)a) \\
&=&(((zx)(((xz)(ya^{2}))y))a^{2})((ax)a)=(((zx)(((a^{2}y)(zx))y))a^{2})((ax)a)
\\
&=&(((zx)((((zx)y)a^{2})y))a^{2})((ax)a)=(((zx)((ya^{2})((zx)y)))a^{2})((ax)a)
\\
&=&(((zx)((y(zx))(a^{2}y)))a^{2})((ax)a)=(((zx)(a^{2}((y(zx))y)))a^{2})((ax)a)
\\
&=&((a^{2}(zx((y(zx))y)))a^{2})((ax)a)\text{.}
\end{eqnarray*}

Then

\begin{eqnarray*}
(f\circ _{0.5}f)(a) &=&\dbigvee \limits_{a=pq}\left \{ f(p)\wedge f(q)\wedge
0.5\right \} \\
&\geq &\left \{ f((a^{2}(zx((y(zx))y)))a^{2})\wedge f((ax)a)\wedge 0.5\right
\} \\
&\geq &\left \{ 
\begin{array}{c}
\{f(a^{2})\wedge f(a^{2})\wedge 0.5\} \\ 
\wedge \{f(a)\wedge f(a)\wedge 0.5\} \wedge 0.5%
\end{array}%
\right \} \\
&\geq &\left \{ 
\begin{array}{c}
\{f(a)\wedge f(a)\wedge 0.5\} \\ 
\wedge \{f(a)\wedge f(a)\wedge 0.5\} \wedge 0.5%
\end{array}%
\right \} \\
&=&\left \{ f(a)\wedge 0.5\right \} =
\end{eqnarray*}

On the other hand by preposition \ref{P 4.5},$\ (f\circ _{0.5}f)(a)\leq f(a)$%
. Hence $(f\circ _{0.5}f)(a)=f(a)$ implies that $f\circ _{0.5}f=f$.

$(i)\Longrightarrow (ii)$

Let $f$ and $g$ be $(\in ,\in \vee q)$-fuzzy bi-ideals of $S$. Then $f\cap
_{0.5}g$ is an $(\in ,\in \vee q)$-fuzzy bi-ideal of $S$. so by $(i)\,$,%
\begin{equation*}
f\cap _{0.5}g=(f\cap _{0.5}g)\circ _{0.5}(f\cap _{0.5}g)\subseteq f\circ
_{0.5}g.
\end{equation*}

On the other hand, by lemma \ref{L 4.4}, $f\circ _{0.5}g$ $\ $and $g\circ
_{0.5}f$ are $(\in ,\in \vee q)$-fuzzy bi-ideals of $S$. Hence $f\circ
_{0.5}g$ $\cap _{0.5}g\circ _{0.5}f$ $\ $is an $(\in ,\in \vee q)$-fuzzy
bi-ideal of $S$. By $(i)$,%
\begin{eqnarray*}
f\circ _{0.5}g\cap _{0.5}g\circ _{0.5}f\ &=&(f\circ _{0.5}g\cap _{0.5}g\circ
_{0.5}f)\circ _{0.5}(f\circ _{0.5}g\cap _{0.5}g\circ _{0.5}f) \\
&\subseteq &(f\circ _{0.5}g)\circ _{0.5}(g\circ _{0.5}f)=((g\circ
_{0.5}f)\circ _{0.5}g)\circ _{0.5}f \\
&\subseteq &((1\circ _{0.5}f)\circ _{0.5}1)\circ _{0.5}f\subseteq (f\circ
_{0.5}1)\circ _{0.5}f \\
&=&f(\text{ as }(f\circ _{0.5}1)\circ _{0.5}f=f\text{, by theorem \ref{T 4.3}%
).}
\end{eqnarray*}

By similar way we can prove that $f\circ _{0.5}g\cap _{0.5}g\circ _{0.5}f\
\subseteq g$. Consequently $\ f\circ _{0.5}g\cap _{0.5}g\circ _{0.5}f\
\subseteq f\cap _{0.5}g$. Therefore%
\begin{equation*}
f\circ _{0.5}g\cap _{0.5}g\circ _{0.5}f\ =f\cap _{0.5}g\text{.}
\end{equation*}
\end{proof}

\section{$(\in ,\in \vee q_{k})$-fuzzy bi-ideals}

All the basic definition concerned with this section in AG-groupoid is same
as for semigroups given in \cite{junn2}. Here we characterize weakly
regular, regular and intra-regular AG-groupoid by the properties of their $%
(\in ,\in \vee q_{k})$-fuzzy bi-ideals.

\begin{definition}
An AG-groupoid $S$ is called weakly regular if for each $a\in S$ we have, $%
a\in (aS)^{2}$, in other words, there exist $x,y\in S$, such that, $%
a=(ax)(ay)$.
\end{definition}

\begin{example}
Let $S=\left \{ 1,2,3,4,5,6\right \} $ be an AG-groupoid with left identity
6 with the following multiplicative table,%
\begin{equation*}
\begin{tabular}{l|llllll}
$\cdot $ & $1$ & $2$ & $3$ & $4$ & $5$ & $6$ \\ \hline
$1$ & $6$ & $1$ & $2$ & $3$ & $4$ & $5$ \\ 
$2$ & $5$ & $6$ & $1$ & $2$ & $3$ & $4$ \\ 
$3$ & $4$ & $5$ & $6$ & $1$ & $2$ & $3$ \\ 
$4$ & $3$ & $4$ & $5$ & $6$ & $1$ & $2$ \\ 
$5$ & $2$ & $3$ & $4$ & $5$ & $6$ & $1$ \\ 
$6$ & $1$ & $2$ & $3$ & $4$ & $5$ & $6$%
\end{tabular}%
\end{equation*}

Clearly $S$ is weakly regular because%
\begin{eqnarray*}
(1\cdot 2)(1\cdot 3) &=&1,(2\cdot 4)(2\cdot 6)=2,(3\cdot 6)(3\cdot 3)=3, \\
(4\cdot 2)(4\cdot 6) &=&4,(5\cdot 4)(5\cdot 3)=5,(6\cdot 6)(6\cdot 6)=6.
\end{eqnarray*}
\end{example}

The following theorems are easy to prove in classical AG-groupoid theory.

\begin{theorem}
\label{th 1}For a weakly regular AG-groupoid $S$ with left identity $e$, the
following conditions are equivalent:

$(i)$ $S$ is regular.

$(ii)$ $R\cap L=RL$ for every right ideal $R$ and every left ideal $L$ of $S$%
.

$(iii)$ $ASA=A$ for every quasi ideal $A$ of $S$.
\end{theorem}

\begin{theorem}
\label{th 2}For a weakly regular AG-groupoid $S$ with left identity $e$, the
following conditions are equivalent:

$(i)$ $S$ is intra-regular.

$(ii)$ $R\cap L=LR$ for every right ideal $R$ and every left ideal $L$ of $S$%
.
\end{theorem}

\begin{theorem}
\label{th 3}For a weakly regular AG-groupoid $S$ with left identity $e$, the
following conditions are equivalent:

$(i)$ $S$ is regular and intra-regular.

$(ii)$ Every quasi ideal of $S$ is idempotent.
\end{theorem}

\begin{definition}
\label{def 3}A fuzzy subset $f$ of an AG-groupoid $S$ is called $(\in ,\in
\vee q_{k})$-fuzzy subgroupoid of $S,$ if we put $\alpha =\in $ and $\beta
=\in \vee q_{k}$ in definition \ref{def 1}.
\end{definition}

\begin{theorem}
\label{th 5}Let $f$ be a fuzzy subset of $S$. Then $f$ is fuzzy subgroupoid
of $S$ if and only if $f(xy)\geq \min \{f(x),f(y),\frac{1-k}{2}\}$.
\end{theorem}

\begin{proof}
It is similar to the proof of theorem $5$ in \cite{junn2}.
\end{proof}

\begin{definition}
A fuzzy subset $f$ of an AG-groupoid $S$ is called $(\in ,\in \vee q_{k})$%
-fuzzy left(right) ideal of $S,$ if for all $x,y\in S$ and $t,r\in (0,1],$
the following condition holds;%
\begin{equation*}
y_{t}\in f\Longrightarrow (xy)_{t}\in \vee q_{k}f(y_{t}\in f\Longrightarrow
(yx)_{t}\in \vee q_{k}f).
\end{equation*}
\end{definition}

\begin{theorem}
A fuzzy subset $f$ of $S$ is an $(\in ,\in \vee q_{k})$-fuzzy left(right)
ideal of $S,$ if and only if 
\begin{equation*}
f(xy)\geq \min \{f(y),\frac{1-k}{2}\} \left( f(xy)\geq \min \{f(x),\frac{1-k%
}{2}\} \right) \text{.}
\end{equation*}
\end{theorem}

\begin{proof}
It is similar to the proof of theorem $8$ in \cite{junn2}.
\end{proof}

Every $(\in ,\in )$-fuzzy left(right) ideal of $S$ is an $(\in ,\in \vee
q_{k})$-fuzzy left(right) ideal of $S$. But an $(\in ,\in \vee q_{k})$-fuzzy
left(right) ideal of $S$ need not be a fuzzy left(right) ideal of $S$.

\begin{theorem}
\label{th 10}If $f$ is an $(\in ,\in \vee q_{k})$-fuzzy left ideal and $g$
is an $(\in ,\in \vee q_{k})$-fuzzy right ideal of a weakly regular
AG-groupoid $S$ with left identity then $f\circ g$ is an $(\in ,\in \vee
q_{k})$-fuzzy two-sided ideal of $S$.
\end{theorem}

\begin{proof}
Let $a,b\in S$, then since $S$ is weakly regular so there exists $x,y\in S$
such that $a=(ax)(ay)$. Now%
\begin{eqnarray*}
(f\circ g)(b)\wedge \frac{1-k}{2} &=&\left \{ \dbigvee
\limits_{b=pq}\{f(p)\wedge g(q)\} \right \} \wedge \frac{1-k}{2} \\
&=&\dbigvee \limits_{b=pq}\left \{ f(p)\wedge g(q)\wedge \frac{1-k}{2}\right
\} \\
&=&\dbigvee \limits_{b=pq}\left \{ f(p)\wedge \frac{1-k}{2}\wedge g(q)\right
\} \text{.}
\end{eqnarray*}

Now using $(1)$,$(2)$, $(3)$ and $(4)$, we have 
\begin{eqnarray*}
ab &=&a(pq)=((ax)(ay))(pq)=((pq)(ay))(ax) \\
&=&((ya)(qp))(ax)=(q((ya)p))(ax) \\
&=&((ax)((ya)p))q=((ax)((ya)(ep)))q \\
&=&((ax)((pe)(ay)))q=((ax)(a((pe)y)))q \\
&=&(a((ax)((pe)y)))q=(a((ax)((ye)p)))q\text{,}
\end{eqnarray*}

since $f$ is an $(\in ,\in \vee q_{k})$-fuzzy left ideal so that%
\begin{eqnarray*}
f(a((ax)((ye)p))) &\geq &f((ax)((ye)p))\wedge \frac{1-k}{2} \\
&\geq &f((ye)p)\wedge \frac{1-k}{2}\wedge \frac{1-k}{2} \\
&\geq &f(p)\wedge \frac{1-k}{2}\wedge \frac{1-k}{2}\wedge \frac{1-k}{2} \\
&=&f(p)\wedge \frac{1-k}{2}\text{.}
\end{eqnarray*}

Thus 
\begin{eqnarray*}
(f\circ g)(b)\wedge \frac{1-k}{2} &=&\dbigvee \limits_{b=pq}\left \{
f(p)\wedge \frac{1-k}{2}\wedge g(q)\right \} \\
&\leq &\dbigvee \limits_{ab=(a((ax)((ye)p)))q}\{f(a((ax)((ye)p)))\wedge
g(q)\} \\
&\leq &\dbigvee \limits_{ab=mn}\{f(m)\wedge g(n)\} \\
&=&(f\circ g)(ab)\text{.}
\end{eqnarray*}%
Therefore $(f\circ g)(ab)\geq (f\circ g)(b)\wedge \frac{1-k}{2}$.

Similarly we can prove that $(f\circ g)(ab)\geq (f\circ g)(a)\wedge \frac{1-k%
}{2}$. Hence $f\circ g$ is an $(\in ,\in \vee q_{k})$-fuzzy two-sided ideal
of $S$.
\end{proof}

\begin{definition}
\label{def 4}A fuzzy subset $f$ of an AG-groupoid $S$ is called $(\in ,\in
\vee q_{k})$-fuzzy generalized bi-ideal of $S,$ if we put $\alpha =\in $ and 
$\beta =\in \vee q_{k}$ in definition \ref{def 2}.
\end{definition}

\begin{definition}
A fuzzy subset $f$ of an AG-groupoid $S$ is called $(\in ,\in \vee q_{k})$%
-fuzzy generalized bi-ideal of $S,$ if definitions \ref{def 3} and \ref{def
4} holds.
\end{definition}

\begin{lemma}
Let $f$ be a fuzzy subset of AG-groupoid $S$, then $f$ is $(\in ,\in \vee
q_{k})$-fuzzy bi-ideal of $S$ if and only if

$(i)$ $f(xy)\geq \min \{f(x),$ $f(y),\frac{1-k}{2}\},$ for all $x,y\in S,$

$(ii)$ $f((xy)z)\geq \min \{f(x),f(z),\frac{1-k}{2}\},$ for all $x,y,z\in S.$
\end{lemma}

\begin{proof}
It is similar to the proof of theorem $5$ in \cite{junn2}.
\end{proof}

\begin{lemma}
Every $(\in ,\in \vee q_{k})$-fuzzy generalized bi-ideal of a weakly regular
AG-groupoid $S$ with left identity $e$, is an $(\in ,\in \vee q_{k})$-fuzzy
bi-ideal of $S$.
\end{lemma}

\begin{proof}
Let $f$ be an $(\in ,\in \vee q_{k})$-fuzzy generalized bi-ideal of $S$ and $%
a,b\in S$, then there exists $x,y\in S$ such that $a=(ax)(ay)$. Then by
using $[4]$, we have $ab=((ax)(ay))b=(a((ax)y)b$. So%
\begin{equation*}
f(ab)=f((a((ax)y)b)\geq \left \{ f(a)\wedge f(b)\wedge \frac{1-k}{2}\right
\} \text{.}
\end{equation*}

This shows that $f$ is an $(\in ,\in \vee q_{k})$-fuzzy subgroupoid of $S$
and so $f$ is an $(\in ,\in \vee q_{k})$-fuzzy bi-ideal of $S$.
\end{proof}

\begin{definition}
A fuzzy subset $f$ is called $(\in ,\in \vee q_{k})$-fuzzy quasi-ideal of
AG-groupoid $S$, if 
\begin{equation*}
f(x)\geq \left \{ (f\circ 1)(x)\wedge (1\circ f)(x)\wedge \frac{1-k}{2}%
\right \} \text{ for all }x\in S.
\end{equation*}
\end{definition}

\begin{lemma}
Every $(\in ,\in \vee q_{k})$-fuzzy quasi-ideal of weakly-regular
AG-groupoid $S$ with left identity $e$, is an $(\in ,\in \vee q_{k})$-fuzzy
bi-ideal of $S$.
\end{lemma}

\begin{proof}
Let $f$ be an $(\in ,\in \vee q)$-fuzzy quasi-ideal of $S$ and $a\in S$.
Since $S$ is weakly regular, so there exists $x,y\in S,$ such that $%
a=(ax)(ay)$, \ Now%
\begin{eqnarray*}
f(ab) &\geq &(f\circ 1)(ab)\wedge (1\circ f)(ab)\wedge \frac{1-k}{2} \\
&=&\left[ \dbigvee \limits_{ab=xy}\{f(x)\wedge 1(y)\} \right] \wedge \left[
\dbigvee \limits_{ab=pq}\{1(p)\wedge f(q)\} \right] \wedge \frac{1-k}{2} \\
&\geq &\{f(a)\wedge 1(b)\} \wedge \{1(a)\wedge f(b)\} \wedge \frac{1-k}{2} \\
&=&\{f(a)\wedge 1\} \wedge \{1\wedge f(b)\} \wedge \frac{1-k}{2} \\
&=&f(a)\wedge f(b)\wedge \frac{1-k}{2}\text{.}
\end{eqnarray*}%
And since by using $(4)$, $(1)$, $(2)$ and $(3)$ we have%
\begin{eqnarray*}
(ab)c &=&(((ax)(ay))b)(ec)=((a((ax)y))b)(ec) \\
&=&((b((ax)y))a)(ec)=(ce)(a(b((ax)y))) \\
&=&a((ce)(b((ax)y)))\text{,}
\end{eqnarray*}

so%
\begin{eqnarray*}
f((ab)c) &\geq &(f\circ 1)((ab)c)\wedge (1\circ f)((ab)c)\wedge \frac{1-k}{2}
\\
&=&\left[ \dbigvee \limits_{(ab)c=a((ce)(b((ax)y)))=xy}\{f(x)\wedge 1(y)\}%
\right] \wedge \left[ \dbigvee \limits_{(ab)c=pq}\{1(p)\wedge f(q)\} \right]
\wedge \frac{1-k}{2} \\
&\geq &[f(a)\wedge 1((ce)(b((ax)y)))\,]\wedge \lbrack 1(ab)\wedge
f(c)]\wedge \frac{1-k}{2} \\
&=&[f(a)\wedge 1]\wedge \lbrack 1\wedge f(c)]\wedge \frac{1-k}{2} \\
&=&f(a)\wedge f(c)\wedge \frac{1-k}{2}\text{.}
\end{eqnarray*}

Thus $f$ is an $(\in ,\in \vee q_{k})$-fuzzy bi-ideal of $S$.
\end{proof}

\begin{definition}
A fuzzy subset $f$ of an AG-groupoid $S$ is called $(\in ,\in \vee q_{k})$%
-fuzzy interior ideal of $S$ if for all $x,y,z\in S$ and $t,r\in (0,1],$ the
following condition holds;

$(i)$ $x_{t}\in f$ and $y_{r}\in f\Longrightarrow (xy)_{t\wedge r}\in \vee
q_{k}f$,

$(ii)$ $y_{t}\in f\Longrightarrow ((xy)z)\in f$.
\end{definition}

The corresponding definition is given in the following theorems.

\begin{theorem}
A fuzzy subset $f$ of an AG-groupoid $S$ is called $(\in ,\in \vee q_{k})$%
-fuzzy interior ideal of $S$ if and only if it satisfies the following
conditions,

$(i)$ $f(xy)\geq \min \{f(x),$ $f(y),$ $\frac{1-k}{2}\},$ for all $x,y\in S,$

$(ii)$ $f((xy)z)\geq \min \{f(y),$ $\frac{1-k}{2}\},$ for all $x,y,z\in S.$
\end{theorem}

\begin{proof}
It is similar to the proof of theorem $5$ in \cite{junn2}.
\end{proof}

\begin{lemma}
Every $(\in ,\in \vee q_{k})$-fuzzy ideal of AG-groupoid $S$ is an $(\in
,\in \vee q_{k})$-fuzzy interior ideal of $S$.
\end{lemma}

\begin{proof}
Let $f$ be an $(\in ,\in \vee q)$-fuzzy ideal of $S$, Then%
\begin{equation*}
f(xy)\geq f(x)\wedge \frac{1-k}{2}\geq f(x)\wedge f(y)\wedge \frac{1-k}{2},
\end{equation*}
so $f$ is fuzzy sub-groupoid of $S$. Also for all $x,a,y\in S$, we have 
\begin{equation*}
f((xa)y)\geq f(xa)\wedge \frac{1-k}{2}\geq f(a)\wedge \frac{1-k}{2}\text{.}
\end{equation*}

Hence $f$ is an $(\in ,\in \vee q_{k})$-fuzzy interior ideal of $S$.
\end{proof}

\begin{definition}
The definition of $f_{k}$, $f\wedge _{k}g$, $f\vee _{k}g$, $f\circ _{k}g$,
for a fuzzy subsets $f$ and $g$ of AG-groupoid $S$ is similar as the
definition $8$\textit{\ }in \cite{junn2}.
\end{definition}

\begin{lemma}
Let $f$ and $g$ be fuzzy subset of AG-groupoid $S$. Then the following holds,

$(i)$ $(f\wedge _{k}g)=(f_{k}\wedge g_{k})$

$(ii)$ $(f\vee _{k}g)=(f_{k}\vee g_{k})$

$(iii)$ $(f\circ _{k}g)=(f_{k}\circ g_{k})$.
\end{lemma}

\begin{proof}
It is similar to the proof of lemma $9$ in \cite{junn2}.
\end{proof}

\begin{lemma}
Let $A$ and $B$ be non-empty subsets of a AG-groupoid $S$, then the
following holds.

$(i)$ $(C_{A}\wedge _{k}C_{B})=(C_{A\cap B})$

$(ii)$ $(C_{A}\vee _{k}C_{B})=(C_{A\cup B})$

$(i)$ $(C_{A}\circ _{k}C_{B})=(C_{AB})$.
\end{lemma}

\begin{lemma}
For an AG-groupoid $S$, the following results holds,

$(i)$ A non-empty subset $L$ of $S$ is a left(right) ideal of $S$ if and
only if $(C_{L})_{k}$ is an $(\in ,\in \vee q_{k})$-fuzzy left(right) ideal
of $S$.

$(ii)$ A non-empty subset $Q$ of $S$ is a quasi-ideal of $S$ if and only if $%
(C_{Q})_{k}$ is an $(\in ,\in \vee q_{k})$-fuzzy quasi-ideal of $S$.

$(iii)$ Let $f$ be an $(\in ,\in \vee q_{k})$-fuzzy left(right) ideal of $S$%
, then $f_{k}$ is a fuzzy left(right) ideal of $S$.
\end{lemma}

\begin{proof}
$(i)$ The proofs of $(i),(ii)$ and $(iii)$ are similar to the proofs of
lemma $11$, lemma $12$ and proposition $1$ in \cite{junn2}.
\end{proof}

\begin{theorem}
\label{th 22}For a weakly regular AG-groupoid $S$ with left identity $e$,
the following condition are equivalent:

$(i)$ $S$ is regular.

$(ii)$ $(f\wedge _{k}g)=(f\circ _{k}g)$ for every $(\in ,\in \vee q_{k})$%
-fuzzy right ideal $f$ and every $(\in ,\in \vee q_{k})$-fuzzy left ideal $g$
of $S$.
\end{theorem}

\begin{proof}
It is similar to the proof of theorem $22$ in \cite{junn2}.
\end{proof}

\begin{theorem}
\label{th 23}For a weakly regular AG-groupoid $S$ with left identity $e$,
the following conditions are equivalent:

$(i)$ $S$ is regular.

$(ii)$ $((f\wedge _{k}g)\wedge _{k}h)\leq ((f\circ _{k}g)\circ _{k}h)$ for
every $(\in ,\in \vee q_{k})$-fuzzy right ideal $f$, every $(\in ,\in \vee
q_{k})$-fuzzy generalized bi-ideal $g$ and every $(\in ,\in \vee q_{k})$%
-fuzzy left ideal $h$ of $S$.

$(iii)$ $((f\wedge _{k}g)\wedge _{k}h)\leq ((f\circ _{k}g)\circ _{k}h)$ for
every $(\in ,\in \vee q_{k})$-fuzzy right ideal $f$, every $(\in ,\in \vee
q_{k})$-fuzzy bi-ideal $g$ and every $(\in ,\in \vee q_{k})$-fuzzy left
ideal $h$ of $S$.

$(iv)$ $((f\wedge _{k}g)\wedge _{k}h)\leq ((f\circ _{k}g)\circ _{k}h)$ for
every $(\in ,\in \vee q_{k})$-fuzzy right ideal $f$, every $(\in ,\in \vee
q_{k})$-fuzzy quasi-ideal $g$ and every $(\in ,\in \vee q_{k})$-fuzzy left
ideal $h$ of $S$.
\end{theorem}

\begin{proof}
$(i)\Longrightarrow (ii)$

Let $f,g$ and $h$ be any $(\in ,\in \vee q_{k})$-fuzzy right ideal, $(\in
,\in \vee q_{k})$-fuzzy generalized bi-ideal and $(\in ,\in \vee q_{k})$%
-fuzzy left ideal of $S$, respectively. Let $a\in S$. Since $S$ is regular,
so there exists $x\in S$ such that $a=(ax)a$. Also since $S$ is weakly
regular,

so there exists $y,z\in S$ such that $a=(ay)(az)$. Therefore by using $(1)$
and $(4)$ we have,%
\begin{eqnarray*}
a &=&(ax)a=(((ay)(az))x)a \\
&=&((x(az))(ay))a=((a(xz))(ay))a \\
&=&((a(xz))(((ay)(az))y))a=((a(xz))((a((ay)z))y))a \\
&=&((a(xz))((y((ay)z))a))a=((a(xz))(((ay)(yz))a))a \\
&=&((a(xz))((((yz)y)a)a))a \\
&=&((a(xz))((((yz)y)((ay)(az)))a))a \\
&=&((a(xz))((((yz)y)(a((ay)z)))a))a \\
&=&((a(xz))((a(((yz)y)((ay)z)))a))a.
\end{eqnarray*}

Now%
\begin{eqnarray*}
((f\circ _{k}g)\circ _{k}h)(a) &=&\left( \dbigvee \limits_{a=pq}\{(f\circ
_{k}g)(p)\wedge h(q)\} \right) \wedge \frac{1-k}{2} \\
&\geq &(f\circ _{k}g)((a(xz))((a(((yz)y)((ay)z)))a))\wedge h(a)\wedge \frac{%
1-k}{2} \\
&\geq &\left( \dbigvee \limits_{(a(xz))((a(((yz)y)((ay)z)))a)=pq}f(p)\wedge
g(q)\right) \wedge h(a)\wedge \frac{1-k}{2} \\
&\geq &(f(a(xz))\wedge g((a(((yz)y)((ay)z)))a))\wedge h(a)\wedge \frac{1-k}{2%
} \\
&\geq &(f(a)\wedge g(a)\wedge \frac{1-k}{2})\wedge h(a)\wedge \frac{1-k}{2}
\\
&=&((f\wedge _{k}g)\wedge _{k}h)(a)\text{.}
\end{eqnarray*}

$(ii)\Longrightarrow (iii)\Longrightarrow (iv)$

$(iv)\Longrightarrow (i)$

Let $f$ and $g$ be any $(\in ,\in \vee q_{k})$-fuzzy right and $(\in ,\in
\vee q_{k})$-fuzzy left ideal of $S$, respectively. Since $1$ is an $(\in
,\in \vee q_{k})$-fuzzy

quasi-ideal of $S$, so by hypothesis, we have%
\begin{eqnarray*}
(f\wedge _{k}g)(a) &=&(f\wedge g)(a)\wedge \frac{1-k}{2} \\
&=&((f\wedge 1)\wedge g)(a)\wedge \frac{1-k}{2} \\
&=&((f\wedge _{k}1)\wedge _{k}g)(a) \\
&\leq &((f\circ _{k}1)\circ _{k}g)(a) \\
&=&((f\circ 1)\circ g)(a)\wedge \frac{1-k}{2} \\
&=&\left( \dbigvee \limits_{a=bc}\{(f\circ 1)(b)\wedge g(c)\} \right) \wedge 
\frac{1-k}{2} \\
&=&\left( \dbigvee \limits_{a=bc}\left \{ \left( \dbigvee
\limits_{a=pq}\{f(p)\wedge 1(q)\} \right) \wedge g(c)\right \} \right)
\wedge \frac{1-k}{2} \\
&=&\left( \dbigvee \limits_{a=bc}\left \{ \left( \dbigvee
\limits_{a=pq}\{f(p)\wedge 1\} \right) \wedge g(c)\right \} \right) \wedge 
\frac{1-k}{2} \\
&=&\left( \dbigvee \limits_{a=bc}\left \{ \left( \dbigvee
\limits_{a=pq}f(p)\right) \wedge g(c)\right \} \right) \wedge \frac{1-k}{2}
\\
&=&\left( \dbigvee \limits_{a=bc}\left \{ \left( \dbigvee
\limits_{a=pq}f(p)\right) \wedge g(c)\right \} \wedge \frac{1-k}{2}\right)
\wedge \frac{1-k}{2} \\
&=&\left( \dbigvee \limits_{a=bc}\left \{ \left( \dbigvee
\limits_{a=pq}\left \{ f(p)\wedge \frac{1-k}{2}\right \} \right) \wedge
g(c)\right \} \right) \wedge \frac{1-k}{2} \\
&\leq &\left( \dbigvee \limits_{a=bc}\left \{ \dbigvee
\limits_{a=pq}f(pq)\wedge g(c)\right \} \right) \wedge \frac{1-k}{2} \\
&=&\left( \dbigvee \limits_{a=bc}\{f(b)\wedge g(c)\} \right) \wedge \frac{1-k%
}{2} \\
&=&(f\circ _{k}g)(a)\text{.}
\end{eqnarray*}

Similarly we can prove that $(f\circ _{k}g)\leq (f\wedge _{k}g)$. Hence $%
(f\circ _{k}g)=(f\wedge _{k}g)$ for every $(\in ,\in \vee q_{k})$-fuzzy
right ideal $f$ and every $(\in ,\in \vee q_{k})$-fuzzy left ideal $g$ of $S$%
. Hence by theorem \ref{th 22}, we get that $S$ is regular.
\end{proof}

\begin{theorem}
\label{th 24}For a weakly regular AG-groupoid $S$ with left identity $e$,
the following conditions are equivalent:

$(i)$ $S$ is regular.

$(ii)$ $f_{k}=((f\circ _{k}1)\circ _{k}f)$ for every $(\in ,\in \vee q_{k})$%
-fuzzy generalized bi-ideal $f$ of $S$.

$(iii)$ $f_{k}=((f\circ _{k}1)\circ _{k}f)$ for every $(\in ,\in \vee q_{k})$%
-fuzzy bi-ideal $f$ of $S$.

$(iv)$ $f_{k}=((f\circ _{k}1)\circ _{k}f)$ for every $(\in ,\in \vee q_{k})$%
-fuzzy quasi-ideal $f$ of $S$.
\end{theorem}

\begin{proof}
$(i)\Longrightarrow (ii)$

Let $f$ be an $(\in ,\in \vee q_{k})$-fuzzy generalized bi-ideal of $S$ and $%
a\in S$. Since $S$ is regular, so there exists $x\in S$ such that $a=(ax)a$.
Therefore we have%
\begin{eqnarray*}
((f\circ _{k}1)\circ _{k}f)(a) &=&((f\circ 1)\circ f)(a)\wedge \frac{1-k}{2}
\\
&=&\left( \dbigvee \limits_{a=yz}\{(f\circ 1)(y)\wedge f(z)\} \right) \wedge 
\frac{1-k}{2} \\
&\geq &(f\circ 1)(ax)\wedge f(a)\wedge \frac{1-k}{2} \\
&=&\left( \dbigvee \limits_{ax=pq}\{f(p)\wedge 1(q)\} \right) \wedge
f(a)\wedge \frac{1-k}{2} \\
&\geq &(f(a)\wedge 1)\wedge f(a)\wedge \frac{1-k}{2} \\
&=&f_{k}(a)\text{.}
\end{eqnarray*}

Thus $((f\circ _{k}1)\circ _{k}f)\geq f_{k}$.

Since $f$ is $(\in ,\in \vee q_{k})$-fuzzy generalized bi-ideal of $S$. So
we have 
\begin{eqnarray*}
((f\circ _{k}1)\circ _{k}f)(a) &=&((f\circ 1)\circ f)(a)\wedge \frac{1-k}{2}
\\
&=&\left( \dbigvee \limits_{a=yz}\{(f\circ 1)(y)\wedge f(z)\} \right) \wedge 
\frac{1-k}{2} \\
&=&\left( \dbigvee \limits_{a=yz}\left \{ \left( \dbigvee
\limits_{y=pq}\{f(p)\wedge 1(q)\} \right) \wedge f(z)\right \} \right)
\wedge \frac{1-k}{2} \\
&=&\left( \dbigvee \limits_{a=yz}\left \{ \left( \dbigvee
\limits_{y=pq}\{f(p)\wedge 1\} \right) \wedge f(z)\right \} \right) \wedge 
\frac{1-k}{2} \\
&=&\left( \dbigvee \limits_{a=yz}\left \{ \dbigvee
\limits_{y=pq}\{f(p)\wedge f(z)\} \right \} \right) \wedge \frac{1-k}{2} \\
&=&\dbigvee \limits_{a=yz}\left \{ \dbigvee \limits_{y=pq}\left \{
(f(p)\wedge f(z))\wedge \frac{1-k}{2}\right \} \wedge \frac{1-k}{2}\right \}
\\
&\leq &\dbigvee \limits_{a=(pq)z}\left \{ f((pq)z)\wedge \frac{1-k}{2}\right
\} \\
&=&f(a)\wedge \frac{1-k}{2} \\
&=&f_{k}(a)\text{.}
\end{eqnarray*}

Thus $((f\circ _{k}1)\circ _{k}f)\leq f_{k}$. Hence $((f\circ _{k}1)\circ
_{k}f)=f_{k}$.

$(ii)\Longrightarrow (iii)\Longrightarrow (iv)$ are obvious.

$(iv)\Longrightarrow (i)$

Let $A$ be any quasi-ideal of $S$. Then $C_{A}$ is an $(\in ,\in \vee q_{k})$%
-fuzzy quasi-ideal of $S$. Hence by hypothesis,%
\begin{equation*}
(C_{A})_{k}=((C_{A}\circ _{k}1)\circ _{k}C_{A})=((C_{A}\circ _{k}C_{S})\circ
_{k}C_{A})=(C_{(AS)A})_{k}\text{.}
\end{equation*}

This implies $A=(AS)A$. Hence it follows from theorem \ref{th 1}, that $S$
is regular.
\end{proof}

\begin{theorem}
\label{th 25}For a weakly regular AG-groupoid $S$ with left identity $e$,
the following conditions are equivalent:

$(i)$ $S$ is regular.

$(ii)$ $(f\wedge _{k}g)=((f\circ _{k}g)\circ _{k}f)$ for every $(\in ,\in
\vee q_{k})$-fuzzy quasi-ideal $f$ and every $(\in ,\in \vee q_{k})$-fuzzy
ideal $g$ of $S$.

$(iii)$ $(f\wedge _{k}g)=((f\circ _{k}g)\circ _{k}f)$ for every $(\in ,\in
\vee q_{k})$-fuzzy quasi-ideal $f$ and every $(\in ,\in \vee q_{k})$-fuzzy
interior ideal $g$ of $S$.

$(iv)$ $(f\wedge _{k}g)=((f\circ _{k}g)\circ _{k}f)$ for every $(\in ,\in
\vee q_{k})$-fuzzy bi-ideal $f$ and every $(\in ,\in \vee q_{k})$-fuzzy
ideal $g$ of $S$.

$(v)$ $(f\wedge _{k}g)=((f\circ _{k}g)\circ _{k}f)$ for every $(\in ,\in
\vee q_{k})$-fuzzy quasi-ideal $f$ and every $(\in ,\in \vee q_{k})$-fuzzy
interior ideal $g$ of $S$.

$(vi)$ $(f\wedge _{k}g)=((f\circ _{k}g)\circ _{k}f)$ for every $(\in ,\in
\vee q_{k})$-fuzzy generalized bi-ideal $f$ and every $(\in ,\in \vee q_{k})$%
-fuzzy ideal $g$ of $S$.

$(vii)$ $(f\wedge _{k}g)=((f\circ _{k}g)\circ _{k}f)$ for every $(\in ,\in
\vee q_{k})$-fuzzy generalized bi-ideal $f$ and every $(\in ,\in \vee q_{k})$%
-fuzzy interior ideal $g$ of $S$.
\end{theorem}

\begin{proof}
$(i)\Longrightarrow (vii)$

Let $f$ is an $(\in ,\in \vee q_{k})$-fuzzy generalized bi-ideal and $g$ is $%
(\in ,\in \vee q_{k})$-fuzzy interior ideal of $S$. Then 
\begin{eqnarray*}
((f\circ _{k}g)\circ _{k}f)(a) &=&((f\circ g)\circ f)(a)\wedge \frac{1-k}{2}
\\
&\leq &((f\circ 1)\circ f)(a)\wedge \frac{1-k}{2} \\
&=&\left( \dbigvee \limits_{a=yz}\{(f\circ 1)(y)\wedge f(z)\} \right) \wedge 
\frac{1-k}{2} \\
&=&\left( \dbigvee \limits_{a=yz}\left \{ \dbigvee
\limits_{y=pq}\{f(p)\wedge 1(q)\} \wedge f(z)\right \} \right) \wedge \frac{%
1-k}{2} \\
&=&\left( \dbigvee \limits_{a=yz}\left \{ \dbigvee
\limits_{y=pq}\{f(p)\wedge 1\} \wedge f(z)\right \} \right) \wedge \frac{1-k%
}{2} \\
&=&\left( \dbigvee \limits_{a=yz}\left \{ \dbigvee
\limits_{y=pq}\{f(p)\wedge f(z)\} \right \} \right) \wedge \frac{1-k}{2} \\
&=&\dbigvee \limits_{a=yz}\left \{ \dbigvee \limits_{y=pq}\{f(p)\wedge
f(z)\wedge \frac{1-k}{2}\} \right \} \\
&=&\dbigvee \limits_{a=(pq)z}\left \{ f(p)\wedge f(z)\wedge \frac{1-k}{2}%
\wedge \frac{1-k}{2}\right \} \\
&\leq &\dbigvee \limits_{a=(pq)z}f((pq)z))\wedge \frac{1-k}{2} \\
&=&f(a)\wedge \frac{1-k}{2} \\
&=&f_{k}(a)\text{.}
\end{eqnarray*}

Therefore $((f\circ _{k}g)\circ _{k}f)(a)\leq f_{k}(a)$. Also 
\begin{eqnarray*}
((f\circ _{k}g)\circ _{k}f)(a) &=&((f\circ g)\circ f)(a)\wedge \frac{1-k}{2}
\\
&\leq &((1\circ g)\circ 1)(a)\wedge \frac{1-k}{2} \\
&=&\left( \dbigvee \limits_{a=yz}\{(1\circ g)(y)\wedge 1(z)\} \right) \wedge 
\frac{1-k}{2} \\
&=&\left( \dbigvee \limits_{a=yz}\left \{ \dbigvee
\limits_{y=pq}\{1(p)\wedge g(q)\} \wedge 1\right \} \right) \wedge \frac{1-k%
}{2} \\
&=&\left( \dbigvee \limits_{a=yz}\left \{ \dbigvee \limits_{y=pq}g(q)\right
\} \right) \wedge \frac{1-k}{2} \\
&=&\dbigvee \limits_{a=yz}\left \{ \dbigvee \limits_{y=pq}\{g(q)\wedge \frac{%
1-k}{2}\} \right \} \\
&\leq &\dbigvee \limits_{a=(pq)z}g((pq)z))\wedge \frac{1-k}{2} \\
&=&g(a)\wedge \frac{1-k}{2}=g_{k}(a)\text{.}
\end{eqnarray*}

Thus $((f\circ _{k}g)\circ _{k}f)\leq (f_{k}\wedge g_{k})=(f\wedge _{k}g)$.
Now let $a\in S$. Since $S$\ is regular so there exists $x\in S$ such \
that, $a=(ax)a.$ Now by using $(1)$ and $(4)$, we have 
\begin{equation*}
a=(ax)a=(((ax)a)x)a=((xa)(ax))a=(a((xa)x))a\text{, }
\end{equation*}

so we have 
\begin{eqnarray*}
((f\circ _{k}g)\circ _{k}f)(a) &=&((f\circ g)\circ f)(a)\wedge \frac{1-k}{2}
\\
&=&\{(f\circ g)(y)\wedge f(z)\} \wedge \frac{1-k}{2} \\
&\geq &(f\circ g)\wedge f(a)\wedge \frac{1-k}{2} \\
&=&\left( \dbigvee \limits_{a((xa)x)=pq}\{f(p)\wedge g(q)\} \right) \wedge
f(a)\wedge \frac{1-k}{2} \\
&\geq &\{f(a)\wedge g(xa)x)\} \wedge f(a)\wedge \frac{1-k}{2} \\
&\geq &\left( f(a)\wedge g(a)\wedge \frac{1-k}{2}\right) \wedge f(a)\wedge 
\frac{1-k}{2} \\
&=&f(a)\wedge g(a)\wedge \frac{1-k}{2} \\
&=&(f\wedge kg)(a)\text{.}
\end{eqnarray*}%
Therefore $((f\circ _{k}g)\circ _{k}f)(a)\geq f_{k}(a)$. Hence $((f\circ
_{k}g)\circ _{k}f)(a)=f_{k}(a)$.

$(vii)\Longrightarrow (v)\Longrightarrow (iii)\Longrightarrow (ii)$ and $%
(vii)\Longrightarrow (vi)\Longrightarrow (iv)\Longrightarrow (ii)$ are clear.

$(ii)\Longrightarrow (i)$

Let $f~$be any $(\in ,\in \vee q_{k})$-fuzzy quasi-ideal of $S$. Then, since 
$S$ itself is an $(\in ,\in \vee q_{k})$-fuzzy two--sided ideal, we have 
\begin{equation*}
f_{k}(a)=f(a)\wedge \frac{1-k}{2}=(f\wedge 1)(a)\wedge \frac{1-k}{2}%
=(f\wedge _{k}1)(a)=(f\circ _{k}1\circ _{k}f)(a)\text{.}
\end{equation*}

Thus by theorem \ref{th 24}, $S$ is regular.
\end{proof}

\begin{theorem}
\label{th 26}For a weakly regular AG-groupoid $S$ with left identity $e,$
the following conditions are equivalent:

$(i)$ $S$ is regular.

$(ii)$ $(f\wedge _{k}g)\leq (f\circ _{k}g)$ for every $(\in ,\in \vee q_{k})$%
-fuzzy quasi-ideal $f$ and every $(\in ,\in \vee q_{k})$-fuzzy left ideal $g$
of $S$.

$(iii)$ $(f\wedge _{k}g)\leq (f\circ _{k}g)$ for every $(\in ,\in \vee
q_{k}) $-fuzzy bi-ideal $f$ and every $(\in ,\in \vee q_{k})$-fuzzy left
ideal $g$ of $S$.

$(iv)$ $(f\wedge _{k}g)\leq (f\circ _{k}g)$ for every $(\in ,\in \vee q_{k})$%
-fuzzy generalized bi-ideal $f$ and every $(\in ,\in \vee q_{k})$-fuzzy left
ideal $g$ of $S$.
\end{theorem}

\begin{proof}
$(i)\Longrightarrow (iv)$

Let $f$ and $g$ be any $(\in ,\in \vee q_{k})$-fuzzy generalized bi-ideal
and any $(\in ,\in \vee q_{k})$-fuzzy left ideal of $S$ respectively. Let $%
a\in S$. Since $S$ is regular, so there exists $x\in S$ such that $a=(ax)a$.
Also since $S$ is weakly regular, so there exists $y,z\in S$ such that $%
a=(ay)(az)$. Now by using $(4)$ and $(1)$ we have%
\begin{eqnarray*}
a &=&(ay)(az)=a((ay)z)=a((((ax)a)y)z) \\
&=&a(((ya)(ax))z)=a((a((ya)x))z)=a((z((ya)x))a)\text{.}
\end{eqnarray*}

So%
\begin{eqnarray*}
(f\circ _{k}g)(a) &=&(f\circ g)(a)\wedge \frac{1-k}{2} \\
&=&\left( \dbigvee \limits_{a=yz}\{f(y)\wedge g(z)\} \right) \wedge \frac{1-k%
}{2} \\
&\geq &f(a)\wedge g((z((ya)x))a)\wedge \frac{1-k}{2} \\
&\geq &f(a)\wedge g(a)\wedge \frac{1-k}{2}\wedge \frac{1-k}{2} \\
&=&f(a)\wedge g(a)\wedge \frac{1-k}{2} \\
&=&(f\wedge _{k}g)(a)\text{.}
\end{eqnarray*}

So $(f\wedge _{k}g)\leq (f\circ _{k}g)$.

$(iv)\Longrightarrow (iii)\Longrightarrow (ii)$ are obvious.

$(ii)\Longrightarrow (i)$

Let $f$ and $g$ be any $(\in ,\in \vee q_{k})$-fuzzy right ideal and any $%
(\in ,\in \vee q_{k})$-fuzzy left ideal of $S$ respectively. Since every $%
(\in ,\in \vee q_{k})$-fuzzy right ideal is an $(\in ,\in \vee q_{k})$-fuzzy
quasi-ideal of $S$. So $(f\wedge _{k}g)\leq (f\circ _{k}g)$. Now%
\begin{eqnarray*}
(f\circ _{k}g)(a) &=&(f\circ g)(a)\wedge \frac{1-k}{2} \\
&=&\left( \dbigvee \limits_{a=yz}\{f(y)\wedge g(z)\} \right) \wedge \frac{1-k%
}{2} \\
&=&\dbigvee \limits_{a=yz}\left( \{f(y)\wedge g(z)\} \wedge \frac{1-k}{2}%
\right) \\
&=&\dbigvee \limits_{a=yz}\left( \left \{ f(y)\wedge \frac{1-k}{2}\right \}
\wedge \left \{ g(z)\wedge \frac{1-k}{2}\right \} \wedge \frac{1-k}{2}\right)
\\
&\leq &\dbigvee \limits_{a=yz}\left( \{f(yz)\wedge g(yz)\} \wedge \frac{1-k}{%
2}\right) \\
&=&f(a)\wedge g(a)\wedge \frac{1-k}{2} \\
&=&(f\wedge _{k}g)(a)\text{.}
\end{eqnarray*}

So $(f\circ _{k}g)\leq (f\wedge _{k}g)$. Thus $(f\wedge _{k}g)=(f\circ
_{k}g) $ for every $(\in ,\in \vee q_{k})$-fuzzy right ideal $f$ and every$%
(\in ,\in \vee q_{k})$-fuzzy left ideal $g$ of $S$. So by theorem \ref{th 22}
we get that $S$ is regular.
\end{proof}

\begin{theorem}
\label{th 27}For a weakly regular AG-groupoid $S$ with left identity$,$ the
following conditions are equivalent:

$(i)$ $S$ is intra-regular.

$(ii)$ $(f\wedge _{k}g)\leq (f\circ _{k}g)$ for every $(\in ,\in \vee q_{k})$%
-fuzzy left ideal $f$ and every $(\in ,\in \vee q_{k})$-fuzzy right ideal $g$
of $S$.
\end{theorem}

\begin{proof}
$(i)\Longrightarrow (ii)$

Let $f$ be an$(\in ,\in \vee q_{k})$-fuzzy left ideal and $g$ be an $(\in
,\in \vee q_{k})$-fuzzy right ideal of $S$. Let $a\in S$. Since $S$ is
intra-regular,

so there exists $x,y\in S$ such that $a=(xa^{2})y$. Also since $S$ is weakly
regular, so there exists $p,q\in S$ such that $a=(ap)(aq)$. Therefore by
using $(4),(1),(2)$ and $(3)$ we have%
\begin{eqnarray*}
a &=&(xa^{2})y=(x(aa))y=(a(xa))y=(y(xa))a \\
&=&(y(xa))(ea)=(ae)((xa)y)=(xa)((ae)y) \\
&=&(xa)((((ap)(aq))e)y)=(xa)(((aq)(ap))y) \\
&=&(xa)((y(ap))(aq))=(xa)(a((y(ap))q))\text{.}
\end{eqnarray*}

So%
\begin{eqnarray*}
(f\circ _{k}g)(a) &=&(f\circ g)(a)\wedge \frac{1-k}{2} \\
&=&\left( \dbigvee \limits_{a=yz}\{f(y)\wedge g(z)\} \right) \wedge \frac{1-k%
}{2} \\
&\geq &\{f(xa)\wedge g(a((y(ap))q))\} \wedge \frac{1-k}{2} \\
&\geq &\left \{ \left( f(a)\wedge \frac{1-k}{2}\right) \wedge \left(
g(a)\wedge \frac{1-k}{2}\right) \right \} \wedge \frac{1-k}{2} \\
&=&f(a)\wedge g(a)\wedge \frac{1-k}{2} \\
&=&(f\wedge _{k}g)(a)\text{.}
\end{eqnarray*}%
Therefore $(f\wedge _{k}g)\leq (f\circ _{k}g)$ for every $(\in ,\in \vee
q_{k})$-fuzzy left ideal $f$ and every $(\in ,\in \vee q_{k})$-fuzzy right
ideal $g$ of $S$.

$(ii)\Longrightarrow (i)$

Let $R$ and $L$ be right and left ideals of $S$, then $(C_{R})_{k}$ and $%
(C_{L})_{k}$ are $(\in ,\in \vee q_{k})$-fuzzy right and $(\in ,\in \vee
q_{k})$-fuzzy left ideals of $S$, respectively. Then by hypothesis, we have, 
$(C_{LR})_{k}=(C_{L}\circ _{k}C_{R})\geq (C_{L}\wedge _{k}C_{R})=(C_{R\cap
L})_{k}$. Thus $R\cap L\subseteq LR.$ Hence it follows from theorem \ref{th
2}, that $S$ is intra-regular$.$
\end{proof}

\begin{theorem}
\label{th 28}For a weakly regular AG-groupoid $S$ with left identity $e,$
the following conditions are equivalent:

$(i)$ $S$ is regular and intra-regular.

$(ii)$ $f\circ _{k}f=f_{k}$ for every $(\in ,\in \vee q_{k})$-fuzzy
quasi-ideal $f$ of $S$.

$(iii)$ $f\circ _{k}f=f_{k}$ for every $(\in ,\in \vee q_{k})$-fuzzy
bi-ideal $f$ of $S$.

$(iv)$ $f\circ _{k}g\geq f\wedge _{k}g$ for every $(\in ,\in \vee q_{k})$%
-fuzzy quasi-ideals $f$ and $g$ of $S$.

$(v)$ $f\circ _{k}g\geq f\wedge _{k}g$ for every $(\in ,\in \vee q_{k})$%
-fuzzy quasi-ideal $f$ and for every $(\in ,\in \vee q_{k})$-fuzzy bi-ideal $%
g$ of $S$.

$(vi)$ $f\circ _{k}g\geq f\wedge _{k}g$ for every $(\in ,\in \vee q_{k})$%
-fuzzy bi-ideals $f$ and $g$ of $S$.
\end{theorem}

\begin{proof}
$(i)\Longrightarrow (vi)$

Let \ $f,g$ be $(\in ,\in \vee q_{k})$-fuzzy bi-ideals of $S$ and $a\in S$.
Since $S$ is regular and intra-regular, so there exists $x,~y,~z\in S,$ such
that $a=(ax)a$ and $a=(ya^{2})z$. Also $S$ is weakly regular so there exists 
$p,q\in S$ such that $a=(ap)(aq)$. Now by using $(4)$ and $(1)$ we have$,$%
\begin{eqnarray*}
a &=&(ax)a=(ax)((ax)a)=(((ap)(aq))x)((ax)a) \\
&=&((a((ap)q))x)((ax)a)=((x((ap)q))a)((ax)a) \\
&=&((x(((ap)(aq))p)q))a)((ax)a) \\
&=&((x((qp)((ap)(aq))))a)((ax)a) \\
&=&((x((qp)(a((ap)q))))a)((ax)a) \\
&=&((x(a((qp)((ap)q))))a)((ax)a) \\
&=&((a(x((qp)((ap)q))))a)((ax)a)\text{.}
\end{eqnarray*}

Then,%
\begin{eqnarray*}
(f\circ _{k}g)(a) &=&(f\circ g)(a)\wedge \frac{1-k}{2} \\
&=&\left( \dbigvee \limits_{a=bc}\{f(b)\wedge g(c)\} \right) \wedge \frac{1-k%
}{2} \\
&\geq &f((a(x((qp)((ap)q))))a)\wedge g((ax)a)\wedge \frac{1-k}{2} \\
&\geq &\left( f(a)\wedge \frac{1-k}{2}\right) \wedge \left( g(a)\wedge \frac{%
1-k}{2}\right) \wedge \frac{1-k}{2} \\
&\geq &\{f(a)\wedge g(a)\} \wedge \frac{1-k}{2} \\
&=&(f\wedge _{k}g)(a)\text{.}
\end{eqnarray*}

Thus $f\circ _{k}g\geq f\wedge _{k}g$ for every $(\in ,\in \vee q_{k})$%
-fuzzy bi-ideals $f$ and $g$ of $S$.

$(vi)\Longrightarrow (v)\Longrightarrow (iv)$ are obvious.

$(iv)\Longrightarrow (ii)$ Put $f=g$ in $(iv)$, we get $f\circ _{k}f\geq
f_{k}$. Since every $(\in ,\in \vee q_{k})$-fuzzy quasi-ideal is an $(\in
,\in \vee q_{k})$-fuzzy subgroupoid, so $f\circ _{k}f\leq f_{k}$. Thus $%
f\circ _{k}f=f_{k}$.

$(vi)\Longrightarrow (iii)$Put $f=g$ in $(iv)$, we get $f\circ _{k}f\geq
f_{k}$. Since every $(\in ,\in \vee q_{k})$-fuzzy bi-ideal is an $(\in ,\in
\vee q_{k})$-fuzzy subgroupoid, so $f\circ _{k}f\leq f_{k}$. Thus $f\circ
_{k}f=f_{k}$.

$(iii)\Longrightarrow (ii)$ is obvious.

$(ii)\Longrightarrow (i)$ Let $Q$ be a quasi ideal of $S$. Then by lemma $4$
in[\cite{junn2}]$,C_{Q}$ is an $(\in ,\in \vee q_{k})$-fuzzy quasi-ideal of $%
S$. Hence by hypothesis, $C_{Q}\circ _{k}C_{Q}=(C_{Q})_{k}$. Thus $%
(C_{QQ})_{k}=C_{Q}\circ _{k}C_{Q}=(C_{Q})_{k}$, implies that $QQ=\dot{Q}$.
So by theorem $\ref{th 3}$, $S$ is both regular and intra-regular.
\end{proof}

\begin{theorem}
For a weakly regular AG-groupoid $S$ with left identity $e,$ the following
conditions are equivalent:

$(i)$ $S$ is regular and intra-regular.

$(ii)$ $(f\circ _{k}g)\wedge (g\circ _{k}f)\geq f\wedge _{k}g$ for every $%
(\in ,\in \vee q_{k})$-fuzzy right ideal $f$ and every $(\in ,\in \vee
q_{k}) $-fuzzy left ideal $g$ of $S$.

$(iii)$ $(f\circ _{k}g)\wedge (g\circ _{k}f)\geq f\wedge _{k}g$ for every $%
(\in ,\in \vee q_{k})$-fuzzy right ideal $f$ and every $(\in ,\in \vee
q_{k}) $-fuzzy quasi-ideal $g$ of $S$.

$(iv)$ $(f\circ _{k}g)\wedge (g\circ _{k}f)\geq f\wedge _{k}g$ for every $%
(\in ,\in \vee q_{k})$-fuzzy right ideal $f$ and every $(\in ,\in \vee
q_{k}) $-fuzzy bi-ideal $g$ of $S$.

$(v)$ $(f\circ _{k}g)\wedge (g\circ _{k}f)\geq f\wedge _{k}g$ for every $%
(\in ,\in \vee q_{k})$-fuzzy right ideal $f$ and every $(\in ,\in \vee
q_{k}) $-fuzzy generalized bi-ideal $g$ of $S$.

$(vi)$ $(f\circ _{k}g)\wedge (g\circ _{k}f)\geq f\wedge _{k}g$ for every $%
(\in ,\in \vee q_{k})$-fuzzy left ideal $f$ and every $(\in ,\in \vee q_{k})$%
-fuzzy quasi-ideal $g$ of $S$.

$(vii)$ $(f\circ _{k}g)\wedge (g\circ _{k}f)\geq f\wedge _{k}g$ for every $%
(\in ,\in \vee q_{k})$-fuzzy left ideal $f$ and every $(\in ,\in \vee q_{k})$%
-fuzzy bi-ideal $g$ of $S$.

$(viii)$ $(f\circ _{k}g)\wedge (g\circ _{k}f)\geq f\wedge _{k}g$ for every $%
(\in ,\in \vee q_{k})$-fuzzy left ideal $f$ and every $(\in ,\in \vee q_{k})$%
-fuzzy generalized bi-ideal $g$ of $S$.

$(ix)$ $(f\circ _{k}g)\wedge (g\circ _{k}f)\geq f\wedge _{k}g$ for every $%
(\in ,\in \vee q_{k})$-fuzzy quasi-ideals $f$ and $g$ of $S$.

$(x)$ $(f\circ _{k}g)\wedge (g\circ _{k}f)\geq f\wedge _{k}g$ for every $%
(\in ,\in \vee q_{k})$-fuzzy quasi-ideal $f$ and every $(\in ,\in \vee
q_{k}) $-fuzzy bi-ideal $g$ of $S$.

$(xi)$ $(f\circ _{k}g)\wedge (g\circ _{k}f)\geq f\wedge _{k}g$ for every $%
(\in ,\in \vee q_{k})$-fuzzy quasi-ideal $f$ and every $(\in ,\in \vee
q_{k}) $-fuzzy generalized bi-ideal $g$ of $S$.

$(xii)$ $(f\circ _{k}g)\wedge (g\circ _{k}f)\geq f\wedge _{k}g$ for every $%
(\in ,\in \vee q_{k})$-fuzzy bi-ideals $f$ and $g$ of $S$.

$(xiii)$ $(f\circ _{k}g)\wedge (g\circ _{k}f)\geq f\wedge _{k}g$ for every $%
(\in ,\in \vee q_{k})$-fuzzy bi-ideal $f$ and every $(\in ,\in \vee q_{k})$%
-fuzzy generalized bi-ideal $g$ of $S$.

$(ixv)$ $(f\circ _{k}g)\wedge (g\circ _{k}f)\geq f\wedge _{k}g$ for every $%
(\in ,\in \vee q_{k})$-fuzzy generalized bi-ideals $f$ and $g$ of $S$.
\end{theorem}

\begin{proof}
$(i)\Longrightarrow (ixv)$

Let $f$ and $g$ be $(\in ,\in \vee q_{k})$-fuzzy generalized bi-ideals of $S$
and $a\in S.$ Since $S$ is regular and intra-regular,

so there exists $x,~y,~z\in S,$ such that $a=(ax)a$ and $a=(ya^{2})z$. In
theorem \ref{T 4.4} it have been shown that $%
a=((a^{2}(zx((y(zx))y)))a^{2})((ax)a)$. then%
\begin{eqnarray*}
(f\circ _{k}g)(a) &=&\left( \dbigvee \limits_{a=bc}\{f(b)\wedge g(c)\} \right)
\wedge \frac{1-k}{2} \\
&\geq &\left \{ f((a^{2}(zx((y(zx))y)))a^{2})\wedge g\right \} \wedge \frac{1-k%
}{2} \\
&\geq &\left \{ f(a^{2})\wedge \frac{1-k}{2}\right \} \wedge \left \{
g(a)\wedge \frac{1-k}{2}\right \} \wedge \frac{1-k}{2} \\
&=&\{f(a^{2})\wedge g(a)\} \wedge \frac{1-k}{2}\text{,}
\end{eqnarray*}

now by using $(1),(4),(2)$ and $(3)$ we have%
\begin{eqnarray*}
f(a^{2}) &=&f(aa)=f(((ax)a)((ax)a)) \\
&=&f((((ax)a)a)(ax))=f(((aa)(ax))(ax)) \\
&=&f((a((aa)x))(ax))=f(((ax)((aa)x))a) \\
&=&f(((ax)((aa)(ex)))a)=f(((ax)((xe)(aa)))a) \\
&=&f(((ax)(a((xe)a)))a)=f((a((ax)((xe)a)))a) \\
&\geq &f(a\dot{)}\text{.}
\end{eqnarray*}

Therefore 
\begin{eqnarray*}
(f\circ _{k}g)(a) &\geq &\{f(a^{2})\wedge g(a)\} \wedge \frac{1-k}{2} \\
&\geq &\{f(a)\wedge g(a)\} \wedge \frac{1-k}{2} \\
&=&(f\wedge _{k}g)(a)\text{.}
\end{eqnarray*}

Similarly we can prove that $(g\circ _{k}f)(a)\geq (f\wedge _{k}g)(a)$.
Hence $(f\circ _{k}g)\wedge (g\circ _{k}f)\geq f\wedge _{k}g$.

$(ixv)\Longrightarrow (xiii)\Longrightarrow (xii)\Longrightarrow
(x)\Longrightarrow (ix)\Longrightarrow (iii)\Longrightarrow (ii)$,

$(ixv)\Longrightarrow (xi)\Longrightarrow (x)$,

$(ixv)\Longrightarrow (viii)\Longrightarrow (iiv)\Longrightarrow
(xi)\Longrightarrow (ii)$ and

$(ixv)\Longrightarrow (v)\Longrightarrow (iv)\Longrightarrow
(iii)\Longrightarrow (ii)$ are obvious.

$(ii)\Longrightarrow (i)$

Let $f$ be an $(\in ,\in \vee q_{k})$-fuzzy right ideal and $g$ be an $(\in
,\in \vee q_{k})$-fuzzy left ideal of $S$. For $a\in S$, we have 
\begin{eqnarray*}
(f\circ _{k}g)(a) &=&(f\circ g)(a)\wedge \\
&=&\left( \dbigvee \limits_{a=yz}\{f(y)\wedge g(z)\} \right) \wedge \frac{1-k%
}{2} \\
&=&\dbigvee \limits_{a=yz}\left \{ f(y)\wedge g(z)\wedge \frac{1-k}{2}\right
\} \\
&=&\dbigvee \limits_{a=yz}\left \{ \left( f(y)\wedge \frac{1-k}{2}\right)
\wedge \left( g(z)\wedge \frac{1-k}{2}\right) \wedge \frac{1-k}{2}\right \}
\\
&\leq &\dbigvee \limits_{a=yz}\left \{ (f(yz)\wedge g(yz))\wedge \frac{1-k}{2%
}\right \} \\
&=&f(a)\wedge g(a)\wedge \frac{1-k}{2} \\
&=&(f\wedge _{k}g)(a)\text{.}
\end{eqnarray*}%
Therefore $(f\circ _{k}g)\leq (f\wedge _{k}g)$. By hypothesis $(f\circ
_{k}g)\geq (f\wedge _{k}g)$, thus $(f\circ _{k}g)=(f\wedge _{k}g)$. Hence by
theorem $\ref{th 22}$, $S$ is regular. Also by hypothesis $(f\circ
_{k}g)\geq (f\wedge _{k}g)$, so by theorem $\ref{th 27}$, $S$ is
intra-regular.
\end{proof}

\end{document}